\documentclass[a4paper, twoside, 11pt, english]{article}

\usepackage[T1]{fontenc}
\usepackage[utf8]{inputenc}

\usepackage{amsmath,amsthm,amsfonts,amssymb,stmaryrd}

\usepackage{newtxtext, courier, euler}
\usepackage[scaled=0.95]{helvet}
\usepackage[cal=euler, scr=boondoxo, scrscaled=1.05]{mathalfa}

\usepackage[numbers,sort&compress]{natbib}
\usepackage{ifthen}
\usepackage{algorithm2e}

\usepackage{fancyhdr, titlesec, url, enumerate, microtype,setspace}

\usepackage[all,knot,poly]{xy}
\usepackage{tikz}
\usepackage{tikz-cd}
\usetikzlibrary{decorations.pathreplacing}

\usepackage{tikz-qtree}
\usepackage{youngtab}
\usepackage{ytableau}

\usepackage[english]{babel}

\usepackage{hyperref}

\usetikzlibrary{calc}

\newcount\hh
\newcount\mm
\mm=\time
\hh=\time
\divide\hh by 60
\divide\mm by 60
\multiply\mm by 60
\mm=-\mm
\advance\mm by \time
\def\hhmm{\number\hh:\ifnum\mm<10{}0\fi\number\mm}
\newcommand{\periodafter}[1]{\ifstrempty{#1}{}{#1.}}
\titleformat{\section}[block]{\scshape\filcenter\LARGE}{\thesection.}{.5em}{}
\titleformat{\subsection}[block]{\bfseries\filcenter\large}{\thesubsection.}{.5em}{\medskip}
\titleformat{\subsubsection}[runin]{\bfseries}{\thesubsubsection.}{.5em}{\periodafter}%{}[.]
\titlespacing{\subsubsection}{0pt}{\topsep}{.5em}

\newtheoremstyle{ntheorem}%
	{\topsep}{\topsep}{\itshape}{0pt}{\bfseries}{.}{.5em}%
	{\thmnumber{#2.\hspace{.5em}}\thmname{#1}\thmnote{ (#3)}}
	
\newtheoremstyle{ndefinition}%
	{\topsep}{\topsep}{\normalfont}{0pt}{\bfseries}{.}{.5em}%
	{\thmnumber{#2.\hspace{.5em}}\thmname{#1}\thmnote{ (#3)}}
	
\newtheoremstyle{nremark}%
	{\topsep}{\topsep}{\normalfont}{0pt}{\itshape}{.}{.5em}%
	{\thmnumber{}\thmname{#1}\thmnote{ (#3)}}

\theoremstyle{ntheorem}
  	\newtheorem{theorem}[subsubsection]{Theorem}
  	\newtheorem{proposition}[subsubsection]{Proposition}
	\newtheorem{lemma}[subsubsection]{Lemma}
  	\newtheorem{corollary}[subsubsection]{Corollary}

\theoremstyle{ndefinition}

	\newtheorem{remark}[subsubsection]{Remark}

\fancypagestyle{plain}{
\fancyhf{}\fancyfoot[LC,RC]{}
\fancyfoot[LE,RO]{$\thepage$}

}

\UseTips
\SelectTips{eu}{11}

\newdir{ >}{{}*!/-10pt/@{>}}
\newdir{ -}{{}*!/-10pt/@{}}
\newdir{> }{{}*!/+10pt/@{>}}

\makeatletter

\xyletcsnamecsname@{dir4{}}{dir{}}
\xydefcsname@{dir4{-}}{\line@ \quadruple@\xydashh@}
\xydefcsname@{dir4{.}}{\point@ \quadruple@\xydashh@}
\xydefcsname@{dir4{~}}{\squiggle@ \quadruple@\xybsqlh@}
\xydefcsname@{dir4{>}}{\Tttip@}
\xydefcsname@{dir4{<}}{\reverseDirection@\Tttip@}

\xydef@\quadruple@#1{%
	\edef\Drop@@{%
		\dimen@=#1\relax
		\dimen@=.5\dimen@
		\A@=-\sinDirection\dimen@
		\B@=\cosDirection\dimen@
		\setboxz@h{%
			\setbox2=\hbox{\kern3\A@\raise3\B@\copy\z@}%
			\dp2=\z@ \ht2=\z@ \wd2=\z@ \box2
			\setbox2=\hbox{\kern\A@\raise\B@\copy\z@}%
			\dp2=\z@ \ht2=\z@ \wd2=\z@ \box2
			\setbox2=\hbox{\kern-\A@\raise-\B@\copy\z@}%
			\dp2=\z@ \ht2=\z@ \wd2=\z@ \box2
			\setbox2=\hbox{\kern-3\A@\raise-3\B@ \noexpand\boxz@}%
			\dp2=\z@ \ht2=\z@ \wd2=\z@ \box2
		}%
		\ht\z@=\z@ \dp\z@=\z@ \wd\z@=\z@ \noexpand\styledboxz@
	}%
}

\xydef@\Tttip@{\kern2pt \vrule height2pt depth2pt width\z@
	\Tttip@@ \kern2pt \egroup
	\U@c=0pt \D@c=0pt \L@c=0pt \R@c=0pt \Edge@c={\circleEdge}%
	\def\Leftness@{.5}\def\Upness@{.5}%
	\def\Drop@@{\styledboxz@}\def\Connect@@{\straight@{\dottedSpread@\jot}}}
	
\xydef@\Tttip@@{%
	\dimen@=.25\dimen@
 	\B@=\cosDirection\dimen@
	\setboxz@h\bgroup\reverseDirection@\line@ \wdz@=\z@ \ht\z@=\z@ \dp\z@=\z@
	{\vDirection@(1,-1)\xydashl@ \xyatipfont\char\DirectionChar}%
	{\vDirection@(1,+1)\xydashl@ \xybtipfont\char\DirectionChar}%
}

\xydef@\ar@form{
	\ifx \space@\next \expandafter\DN@\space{\xyFN@\ar@form}%
	\else\ifx ^\next \DN@ ^{\xyFN@\ar@style}\edef\arvariant@@{\string^}%
	\else\ifx _\next \DN@ _{\xyFN@\ar@style}\edef\arvariant@@{\string_}%
	\else\ifx 0\next \DN@ 0{\xyFN@\ar@style}\def\arvariant@@{0}%
	\else\ifx 1\next \DN@ 1{\xyFN@\ar@style}\def\arvariant@@{1}%
	\else\ifx 2\next \DN@ 2{\xyFN@\ar@style}\def\arvariant@@{2}%
	\else\ifx 3\next \DN@ 3{\xyFN@\ar@style}\def\arvariant@@{3}%
	\else\ifx 4\next \DN@ 4{\xyFN@\ar@style}\def\arvariant@@{4}%
	\else\ifx \bgroup\next \let\next@=\ar@style
	\else\ifx [\next \DN@[##1]{\ar@modifiers{[##1]}}%]
	\else\ifx *\next \DN@ *{\ar@modifiers}%
	\else\addLT@\ifx\next \let\next@=\ar@slide
	\else\ifx /\next \let\next@=\ar@curveslash
	\else\ifx (\next \let\next@=\ar@curveinout %)
	\else\addRQ@\ifx\next \addRQ@\DN@{\ar@curve@}%
	\else\addLQ@\ifx\next \addLQ@\DN@{\xyFN@\ar@curve}%
	\else\addDASH@\ifx\next \addDASH@\DN@{\defarstem@-\xyFN@\ar@}%
	\else\addEQ@\ifx\next \addEQ@\DN@{\def\arvariant@@{2}\defarstem@-\xyFN@\ar@}%
	\else\addDOT@\ifx\next \addDOT@\DN@{\defarstem@.\xyFN@\ar@}%
	\else\ifx :\next \DN@:{\def\arvariant@@{2}\defarstem@.\xyFN@\ar@}%
	\else\ifx ~\next \DN@~{\defarstem@~\xyFN@\ar@}%
	\else\ifx !\next \DN@!{\dasharstem@\xyFN@\ar@}%
	\else\ifx ?\next \DN@?{\ar@upsidedown\xyFN@\ar@}%
	\else \let\next@=\ar@error
	\fi\fi\fi\fi\fi\fi\fi\fi\fi\fi\fi\fi\fi\fi\fi\fi\fi\fi\fi\fi\fi\fi\fi \next@}

\makeatother

\tikzset{%
    triple line/.code={\tikzset{%
            double equal sign distance, % replace by double distance = 'measure' 
            double=\pgfkeysvalueof{/tikz/commutative diagrams/background color}}},
    quadruple line/.code={\tikzset{%
            double equal sign distance, % replace by double distance = 'measure'
            double=\pgfkeysvalueof{/tikz/commutative diagrams/background color}}},
    Rrightarrow/.code={\tikzcdset{triple line}\pgfsetarrows{tikzcd implies cap-tikzcd implies}},
    RRightarrow/.code={\tikzcdset{quadruple line}\pgfsetarrows{tikzcd implies cap-tikzcd implies}}
}

\def\ie{{\emph{i.e.}~}}

\newcommand{\set}[1]{\{#1\}}

\newcommand{\dual}[1]{\overline{#1}}

\definecolor{mred}{rgb}{0.7,0.1,0.1}
\definecolor{mblue}{rgb}{0,0,0.8}
\definecolor{mgreen}{rgb}{0,0.6,0.3}

\newcommand\Fact[1]{{F}#1}

\newcommand\map[3]{#1: #2 \rightarrow #3}

\newcommand\gpo[1]{((#1),\mu,\eta, (-)^{-1})}

\newcommand\fcat[1]{\overrightarrow{\boldsymbol\Pi}(#1)}
\newcommand\trace[3]{\overrightarrow{\mathfrak{T}}(#1)(#2,#3)}

\newcommand\sysh[2]{\overrightarrow{H}_{#1}(#2)}

\newcommand\sysp[2]{\overrightarrow{P}_{\!\!#1}(#2)}

\newcommand\Cr{\mathcal{C}}

\newcommand\Er{\mathcal{E}}
\newcommand\Br{\mathcal{B}}

\newcommand\M[0]{\mathcal{M}}

\newcommand\N[0]{\mathcal{N}}

\newcommand\lF[0]{\mathcal{F}}
\newcommand\V{\mathbf{V}}

\newcommand\twoB[0]{\Br^{2}}
\newcommand\A[0]{\mathcal{A}}

\newcommand\diP[1]{\overrightarrow{\mathbf{P}}\!(#1)}

\newcommand\Kernel[1]{ Ker(#1) }
\newcommand\Crok[1]{Cok(#1)}

\newcommand{\tr}[1]{\overline{#1}}
\newcommand{\fl}{\rightarrow}
\newcommand{\dfl}{\Rightarrow}
\newcommand{\Xr}{\mathcal{X}}
\newcommand{\Yr}{\mathcal{Y}}
\newcommand{\Ar}{\mathcal{A}}

\newcommand{\dcyl}[1]{#1 \times\!\! \uparrow\!\! I}

\def\catego#1{{\bf{\sf #1}}}

\newcommand\cat{\catego{Cat}}
\newcommand\catsigmab[0]{\cat_{\Br_0}/ \Br}
\newcommand\catsigma[0]{\cat_{X}}
\newcommand\catxp[0]{{\cat_{X}/ \diP{\Xr}}}

\newcommand\ab{\catego{Ab}}
\newcommand\gp{\catego{Gp}}
\newcommand\gpos[1]{\gp(#1)}
\newcommand\abo[1]{\ab(#1)}
\newcommand\act{\catego{Act}}
\newcommand\setcat{\catego{Set}}
\newcommand\psetcat{\setcat_{*}}
\newcommand\natsys[2]{\catego{NatSys}(#1,#2)}
\newcommand\natsysnu[2]{\catego{NatSys}_{\nu}(#1,#2)}
\newcommand\laxsys[2]{\catego{LaxSys}(#1,#2)}
\newcommand\opnat[1]{\catego{opNat}(#1)}
\newcommand\opnatnu[1]{\catego{opNat}_{\nu}(#1)}
\renewcommand{\top}{\catego{Top}}

\newcommand{\dtop}{\catego{dTop}}

\begin{document}
\thispagestyle{empty}

\begin{center}

% Titre
\begin{doublespace}
\begin{huge}
{\scshape Time-reversal homotopical properties}

{\scshape of concurrent systems}
\end{huge}

\bigskip
\hrule height 1.5pt 
\bigskip

% Auteurs
\begin{Large}
{\scshape Cameron Calk \qquad Eric Goubault \qquad Philippe Malbos}
\end{Large}
\end{doublespace}

%%%
\vspace{3cm}

% Résumé, mots-clefs et classification
\begin{small}\begin{minipage}{14cm}
\noindent\textbf{Abstract --}
Directed topology was introduced as a model of concurrent programs, where the flow of time is described by distinguishing certain paths in the topological space representing such a program. Algebraic invariants which respect this directedness have been introduced to classify directed spaces. In this work we study the properties of such invariants with respect to the reversal of the flow of time in directed spaces. Known invariants, natural homotopy and homology, have been shown to be unchanged under this time-reversal. We show that these can be equipped with additional algebraic structure witnessing this reversal. Specifically, when applied to a directed space and to its reversal, we show that these refined invariants yield dual objects. We further refine natural homotopy by introducing a notion of relative directed homotopy and showing the existence of a long exact sequence of natural homotopy systems.

\medskip

 Natural homology and natural homotopy.

\smallskip\noindent\textbf{M.S.C. 2010 --} 68Q85, 18D35, 55U99.
\end{minipage}\end{small}

%%%
\vspace{1cm}

%%% Table des matières
\begin{small}\begin{minipage}{12cm}
\renewcommand{\contentsname}{}
\setcounter{tocdepth}{2}
\tableofcontents
\end{minipage}
\end{small}
\end{center}

\clearpage

\section{Introduction}

\subsection{Time-reversal properties of concurrent systems}

Directed topology was originally introduced as a model, and a tool, for studying
and classifying concurrent systems in computer science \cite{Pratt,CONCUR92}. 
In this approach, the possible states of several processes running concurrently are modeled as points in a topological space of configurations, in which executions are described by paths. Restricted areas appear when these processes have to synchronize, to perform a joint task, or to use a shared object that cannot be shared by more than a certain number of processes.
It is natural to study the homotopical properties of this configuration space in order to deduce some interesting properties
of the parallel programs involved, for verification purposes, or for classifying 
synchronization primitives. A usual model for concurrent processes is actually that of higher-dimensional automata, which are based on (pre-)cubical sets, and are the most expressive known models in concurrency theory \cite{Glabbeek}. 
But contrarily to ordinary algebraic topology, the invariants of interest are invariants
under some form of continuous deformation, but which has to respect the flow of time. 
In short, the only valid homotopies are the ones which never invert the flow of time. 
For mathematical developments and some applications we refer the reader to the two books 
\cite{grandisbook,fajstrup16}. 

\subsubsection*{Directed spaces and concurrent programs}
Directed topological invariants, most notably the computationally tractable ones such
as homology, have been long in the making (starting again with \cite{CONCUR92}). Most
directed homology theories have proven too weak to classify essential features of 
directed topology, until the proposal \cite{Eilenberg,naturalhomology}. Let us 
review quickly the main idea from \cite{naturalhomology}. 
Recall from~\cite{grandisbook} that a \emph{directed space}, or a \emph{dispace} for short, is a pair $\Xr=(X, dX)$, where $X$ is a topological space and $dX$ is a set of \emph{paths} in $X$, i.e., continuous maps from~$[0,1]$ to~$X$, called \emph{directed paths}, of \emph{dipaths} for short, such that every constant path is directed, and such that $dX$ is closed under monotonic reparametrization and concatenation.

Partially-ordered spaces, or pospaces, form particular dispaces~: these are topological spaces $X$
equipped with a partial order $\leq$ on $X$ which is closed under
the product topology. The directed structure is thus given by paths $p~: [0,1]
\rightarrow X$ such that $p(s) \leq p(t)$, for all $s \leq t$ in $[0,1]$. 
Another useful class of dispaces is given by the directed geometric realization
of finite precubical sets, see e.g. \cite{frgatc}. These are made of gluings of
cubical cells, on which the dispace structure is locally that of a particular partially-ordered space~: each $n$-dimensional cell is identified with $[0,1]^n$ ordered componentwise. This last
class is in particular very useful in applications to concurrency and distributed
systems theory, see e.g. \cite{fajstrup16}. 

As an example, we have depicted two dispaces in Figure \ref{F:ExamplesNon-EquivalentConcurrentPrograms}, which are built as the gluing of squares (the white ones), each of which is equipped with the product ordering of $\mathbb{R}^2$. 
\begin{figure}[h]
\centering
\begin{tikzpicture}[auto,scale = 0.8]
\node (al) at (0,0) {$\bullet$};
\node (be) at (4,4) {$\bullet$};
\node (alp) at (-0.3,-0.3) {\scriptsize{$\alpha_X$}};
\node (bet) at (4.3,4.3) {\scriptsize{$\beta_X$}};
\node (a) at (1.8,-0.8) {\scriptsize{$\Xr = (X,dX)$}};
\draw (0,0) rectangle (1.5,1.5);
\draw (1.5,1.5) rectangle (4,4);
\draw [fill = gray!50,draw = gray!50] (0.5,0.5) rectangle (1,1);
\draw [fill = gray!50,draw = gray!50] (2,3) rectangle (2.5,3.5);
\draw [fill = gray!50,draw = gray!50] (3,2) rectangle (3.5,2.5);
\draw [dotted] (1,0) to (1,1.5);
\draw [dotted] (0.5,0) to (0.5,1.5);
\draw [dotted] (0,1) to (1.5,1);
\draw [dotted] (0,0.5) to (1.5,0.5);
\draw [dotted] (2,1.5) to (2,4);
\draw [dotted] (2.5,1.5) to (2.5,4);
\draw [dotted] (3,1.5) to (3,4);
\draw [dotted] (3.5,1.5) to (3.5,4);
\draw [dotted] (1.5,2) to (4,2);
\draw [dotted] (1.5,2.5) to (4,2.5);
\draw [dotted] (1.5,3) to (4,3);
\draw [dotted] (1.5,3.5) to (4,3.5);
\draw (0.75,-0.1) to (0.75,0.1);
\draw (-0.1,0.75) to (0.1,0.75);
\draw (2.25,1.4) to (2.25,1.6);
\draw (3.25,1.4) to (3.25,1.6);
\draw (1.4,2.25) to (1.6,2.25);
\draw (1.4,3.25) to (1.6,3.25);
\node (S11) at (0.75,-0.3) {\footnotesize{S}};
\node (U21) at (-0.3,0.75) {\footnotesize{U}};
\node (U11) at (2.25,1.2) {\footnotesize{U}};
\node (S12) at (3.25,1.2) {\footnotesize{S}};
\node (U22) at (1.2,2.25) {\footnotesize{U}};
\node (S21) at (1.2,3.25) {\footnotesize{S}};
\draw [thick] (0,0) .. controls (0.3,1.5)  .. (1.5,1.5);
\draw [thick] (0,0) .. controls (1.5,0.3)  .. (1.5,1.5);
\draw [thick] (1.5,1.5) to (4,4);
\draw [thick] (1.5,1.5) .. controls (1.6,4)  .. (4,4);
\draw [thick] (1.5,1.5) .. controls (4,1.6)  .. (4,4);
\end{tikzpicture}
\begin{tikzpicture}
\draw[color = white] (0,-1) rectangle (0.5,-1.1);
\end{tikzpicture}
\begin{tikzpicture}[auto,scale = 0.85]
\node (al) at (0,0) {$\bullet$};
\node (be) at (2.5,2.5) {$\bullet$};
\node (alp) at (-0.3,-0.3) {\scriptsize{$\alpha_Y$}};
\node (bet) at (2.8,2.8) {\scriptsize{$\beta_Y$}};
\node (a) at (1.3,-0.8) {\scriptsize{$\mathcal{Y} = (Y,dY)$}};
\draw (0,0) rectangle (2.5,2.5);
\draw [fill = gray!50,draw = gray!50] (0.5,0.5) rectangle (1,1);
\draw [fill = gray!50,draw = gray!50] (0.5,1.5) rectangle (1,2);
\draw [fill = gray!50,draw = gray!50] (1.5,0.5) rectangle (2,1);
\draw [fill = gray!50,draw = gray!50] (1.5,1.5) rectangle (2,2);
\draw [dotted] (1,0) to (1,2.5);
\draw [dotted] (0.5,0) to (0.5,2.5);
\draw [dotted] (1.5,0) to (1.5,2.5);
\draw [dotted] (2,0) to (2,2.5);
\draw [dotted] (0,1) to (2.5,1);
\draw [dotted] (0,0.5) to (2.5,0.5);
\draw [dotted] (0,1.5) to (2.5,1.5);
\draw [dotted] (0,2) to (2.5,2);
\draw (0.75,-0.1) to (0.75,0.1);
\draw (-0.1,0.75) to (0.1,0.75);
\draw (1.75,-0.1) to (1.75,0.1);
\draw (-0.1,1.75) to (0.1,1.75);
\node (S11) at (0.75,-0.3) {\footnotesize{S}};
\node (U21) at (-0.3,0.75) {\footnotesize{U}};
\node (S12) at (1.75,-0.3) {\footnotesize{S}};
\node (U22) at (-0.3,1.75) {\footnotesize{U}};
\draw [thick] (0,0) .. controls (0.2,1.25)  .. (1.25,1.25);
\draw [thick] (0,0) .. controls (1.25,0.2)  .. (1.25,1.25);
\draw [thick] (2.5,2.5) .. controls (1.25,2.3)  .. (1.25,1.25);
\draw [thick] (2.5,2.5) .. controls (2.3,1.25)  .. (1.25,1.25);
\draw [thick] (0,0) .. controls (0.1,2.5)  .. (2.5,2.5);
\draw [thick] (0,0) .. controls (2.5,0.1)  .. (2.5,2.5);
\node (b) at (0,-1) {};
\end{tikzpicture}
\begin{tikzpicture}
\draw[color = white] (0,-1) rectangle (0.5,-1.1);
\end{tikzpicture}
\caption{Examples of pospaces coming from non-equivalent concurrent programs.}
\label{F:ExamplesNon-EquivalentConcurrentPrograms}
\end{figure}
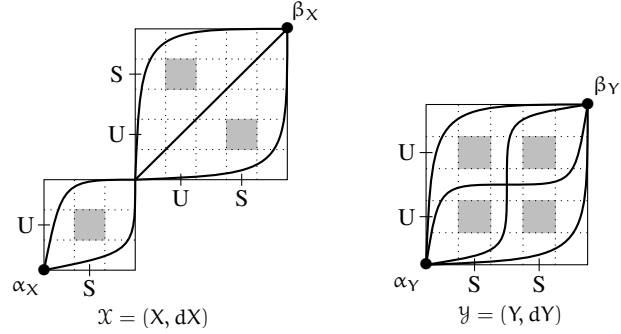
They are the directed geometric realization 
of certain precubical sets, as we mentioned above, 
i.e. are higher-dimensional automata in the sense of~\cite{Pratt}. 
They are not dihomeomorphic spaces since they are already non homotopy equivalent: the fundamental group of 
the leftmost one, that we call $X$, is the free abelian group on three generators, whereas the fundamental group of
the rightmost one, that we call $Y$, is the free abelian group on four generators. 
Consider now the topological space of dipaths, with the compact-open topology,
from the lowest point of $X$ (resp. $Y$), which
we denote by $\alpha_X$ (resp. $\alpha_Y$), to the highest point of $X$ (resp. $Y$), which we denote by~$\beta_X$ (resp. $\alpha_Y$). 
The topological space $\overrightarrow{Di}(\Xr)(\alpha_X,\beta_X)$ of directed paths
from $\alpha_X$ to $\beta_X$, is homotopy equivalent to a six point space, corresponding to the six dihomotopy classes of dipaths pictured in Figure~\ref{F:ExamplesNon-EquivalentConcurrentPrograms}. The topological space $\overrightarrow{Di}(\Yr)(\alpha_Y,\beta_Y)$ is also homotopy
equivalent to a six point space, corresponding again to the six dihomotopy classes of dipaths pictured
in Figure \ref{F:ExamplesNon-EquivalentConcurrentPrograms}. 
However, these two dispaces should not be considered as equivalent, in the sense that they correspond to distinct concurrent programs. Therefore comparing spaces of dipaths exclusively between two particular points in each space is not sufficient for distinguishing these dispaces.

\subsubsection*{Natural homotopy}
The main idea of \cite{naturalhomology} is to encode how the homotopy types of the spaces of directed paths vary when we move the end points. With the possibility to consider all the directed path spaces,  we can distinguish the two former pospaces because if we consider the space of directed paths between $\alpha_X$ and $\beta'_X$, as in Figure \ref{changebasepoints}, it has the homotopy type of a discrete space with four points. Furthermore, we can show that in $Y$, there is no pair of points between which we have a directed path space with that homotopy type.
The algebraic structure which logs all of the homotopy types of the directed
path spaces between each pair of points is that of a natural system, see Section \ref{SS:NaturalSystemsWithCompositionPairing}. 
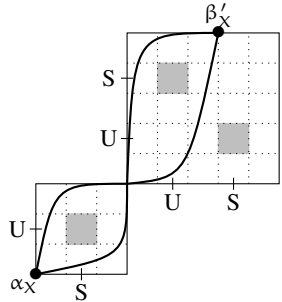
\begin{figure}[h]
\centering
\begin{tikzpicture}[auto,scale = 0.8]
\draw (0,0) rectangle (1.5,1.5);
\draw (1.5,1.5) rectangle (4,4);
\draw [fill = gray!50,draw = gray!50] (0.5,0.5) rectangle (1,1);
\draw [fill = gray!50,draw = gray!50] (2,3) rectangle (2.5,3.5);
\draw [fill = gray!50,draw = gray!50] (3,2) rectangle (3.5,2.5);
\draw [dotted] (1,0) to (1,1.5);
\draw [dotted] (0.5,0) to (0.5,1.5);
\draw [dotted] (0,1) to (1.5,1);
\draw [dotted] (0,0.5) to (1.5,0.5);
\draw [dotted] (2,1.5) to (2,4);
\draw [dotted] (2.5,1.5) to (2.5,4);
\draw [dotted] (3,1.5) to (3,4);
\draw [dotted] (3.5,1.5) to (3.5,4);
\draw [dotted] (1.5,2) to (4,2);
\draw [dotted] (1.5,2.5) to (4,2.5);
\draw [dotted] (1.5,3) to (4,3);
\draw [dotted] (1.5,3.5) to (4,3.5);
\draw (0.75,-0.1) to (0.75,0.1);
\draw (-0.1,0.75) to (0.1,0.75);
\draw (2.25,1.4) to (2.25,1.6);
\draw (3.25,1.4) to (3.25,1.6);
\draw (1.4,2.25) to (1.6,2.25);
\draw (1.4,3.25) to (1.6,3.25);
\node (S11) at (0.75,-0.3) {\footnotesize{S}};
\node (U21) at (-0.3,0.75) {\footnotesize{U}};
\node (U11) at (2.25,1.2) {\footnotesize{U}};
\node (S12) at (3.25,1.2) {\footnotesize{S}};
\node (U22) at (1.2,2.25) {\footnotesize{U}};
\node (S21) at (1.2,3.25) {\footnotesize{S}};
\draw [thick] (0,0) .. controls (0.3,1.5)  .. (1.5,1.5);
\draw [thick] (0,0) .. controls (1.5,0.3)  .. (1.5,1.5);
\draw [thick] (1.5,1.5) .. controls (1.6,4)  .. (3,4);
\draw [thick] (1.5,1.5) .. controls (2.5,1.6)  .. (3,4);
\node (al) at (0,0) {$\bullet$};
\node (be) at (3,4) {$\bullet$};
\node (alp) at (-0.2,-0.2) {\scriptsize{$\alpha_X$}};
\node (bet) at (3,4.3) {\scriptsize{$\beta'_X$}};
\end{tikzpicture}
\caption{Changing the base points to exhibit a particular space of directed paths.}\label{changebasepoints}
\end{figure}

\subsubsection*{Time reversal invariance}
Now, let us consider the concurrent program semantics, and its model as a directed
space $\Xr=(X,dX)$, pictured in Figure \ref{fig:timereversal}, and invert the time flow. If we orient the time flow from left to right and from bottom to top, we need to rotate its representation as a dispace, as shown right of Figure~\ref{fig:timereversal}. As concurrent processes, these two programs should not be considered as equivalent under any form of well accepted equivalence. These two concurrent programs actually have equivalent prime event structure representations, see \cite{CAT0}, that are not bisimulation equivalent \cite{Bisimulation} under
any kind of sensible bisimulation.
L. Fajstrup and K. Hess noted that natural homotopy and homology theories do not
distinguish between these two cases, but produce isomorphic natural systems. 

It is one purpose of this paper to show that natural homotopy and homology theories
lack an algebraic ingredient, a form of composition, that will make the invariants associated to these two dispaces non-isomorphic, but rather dual to each other. 
This composition was introduced by Porter \cite{Porter2} in order to link natural systems seen as coefficients of generalized cohomology theories \cite{BauesWirsching} to coefficients for Quillen cohomology theories \cite{quillencohomology}. 
This is explained, for completeness purposes, in the
first part of this article. 

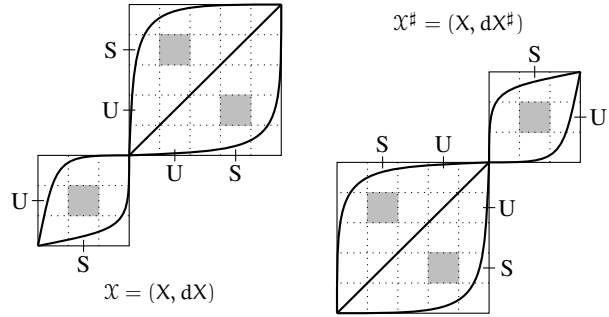
\begin{figure}[h]
\centering
\begin{tikzpicture}[auto,scale = 0.8]
\node (a) at (2,-0.8) {\scriptsize{$\Xr = (X,dX)$}};
\draw (0,0) rectangle (1.5,1.5);
\draw (1.5,1.5) rectangle (4,4);
\draw [fill = gray!50,draw = gray!50] (0.5,0.5) rectangle (1,1);
\draw [fill = gray!50,draw = gray!50] (2,3) rectangle (2.5,3.5);
\draw [fill = gray!50,draw = gray!50] (3,2) rectangle (3.5,2.5);
\draw [dotted] (1,0) to (1,1.5);
\draw [dotted] (0.5,0) to (0.5,1.5);
\draw [dotted] (0,1) to (1.5,1);
\draw [dotted] (0,0.5) to (1.5,0.5);
\draw [dotted] (2,1.5) to (2,4);
\draw [dotted] (2.5,1.5) to (2.5,4);
\draw [dotted] (3,1.5) to (3,4);
\draw [dotted] (3.5,1.5) to (3.5,4);
\draw [dotted] (1.5,2) to (4,2);
\draw [dotted] (1.5,2.5) to (4,2.5);
\draw [dotted] (1.5,3) to (4,3);
\draw [dotted] (1.5,3.5) to (4,3.5);
\draw (0.75,-0.1) to (0.75,0.1);
\draw (-0.1,0.75) to (0.1,0.75);
\draw (2.25,1.4) to (2.25,1.6);
\draw (3.25,1.4) to (3.25,1.6);
\draw (1.4,2.25) to (1.6,2.25);
\draw (1.4,3.25) to (1.6,3.25);
\node (S11) at (0.75,-0.3) {\footnotesize{S}};
\node (U21) at (-0.3,0.75) {\footnotesize{U}};
\node (U11) at (2.25,1.2) {\footnotesize{U}};
\node (S12) at (3.25,1.2) {\footnotesize{S}};
\node (U22) at (1.2,2.25) {\footnotesize{U}};
\node (S21) at (1.2,3.25) {\footnotesize{S}};
\draw [thick] (0,0) .. controls (0.3,1.5)  .. (1.5,1.5);
\draw [thick] (0,0) .. controls (1.5,0.3)  .. (1.5,1.5);
\draw [thick] (1.5,1.5) to (4,4);
\draw [thick] (1.5,1.5) .. controls (1.6,4)  .. (4,4);
\draw [thick] (1.5,1.5) .. controls (4,1.6)  .. (4,4);
\end{tikzpicture}
\begin{tikzpicture}
\draw[color = white] (0,-1) rectangle (0.5,-1.1);
\end{tikzpicture}
\begin{tikzpicture}[auto,scale = 0.8,rotate=180]
\node (a) at (2,-0.8) {\scriptsize{$\Xr^\sharp = (X,dX^\sharp)$}};
\draw (0,0) rectangle (1.5,1.5);
\draw (1.5,1.5) rectangle (4,4);
\draw [fill = gray!50,draw = gray!50] (0.5,0.5) rectangle (1,1);
\draw [fill = gray!50,draw = gray!50] (2,3) rectangle (2.5,3.5);
\draw [fill = gray!50,draw = gray!50] (3,2) rectangle (3.5,2.5);
\draw [dotted] (1,0) to (1,1.5);
\draw [dotted] (0.5,0) to (0.5,1.5);
\draw [dotted] (0,1) to (1.5,1);
\draw [dotted] (0,0.5) to (1.5,0.5);
\draw [dotted] (2,1.5) to (2,4);
\draw [dotted] (2.5,1.5) to (2.5,4);
\draw [dotted] (3,1.5) to (3,4);
\draw [dotted] (3.5,1.5) to (3.5,4);
\draw [dotted] (1.5,2) to (4,2);
\draw [dotted] (1.5,2.5) to (4,2.5);
\draw [dotted] (1.5,3) to (4,3);
\draw [dotted] (1.5,3.5) to (4,3.5);
\draw (0.75,-0.1) to (0.75,0.1);
\draw (-0.1,0.75) to (0.1,0.75);
\draw (2.25,1.4) to (2.25,1.6);
\draw (3.25,1.4) to (3.25,1.6);
\draw (1.4,2.25) to (1.6,2.25);
\draw (1.4,3.25) to (1.6,3.25);
\node (S11) at (0.75,-0.3) {\footnotesize{S}};
\node (U21) at (-0.3,0.75) {\footnotesize{U}};
\node (U11) at (2.25,1.2) {\footnotesize{U}};
\node (S12) at (3.25,1.2) {\footnotesize{S}};
\node (U22) at (1.2,2.25) {\footnotesize{U}};
\node (S21) at (1.2,3.25) {\footnotesize{S}};
\draw [thick] (0,0) .. controls (0.3,1.5)  .. (1.5,1.5);
\draw [thick] (0,0) .. controls (1.5,0.3)  .. (1.5,1.5);
\draw [thick] (1.5,1.5) to (4,4);
\draw [thick] (1.5,1.5) .. controls (1.6,4)  .. (4,4);
\draw [thick] (1.5,1.5) .. controls (4,1.6)  .. (4,4);
\end{tikzpicture}
\begin{tikzpicture}
\draw[color = white] (0,-1) rectangle (0.5,-1.1);
\end{tikzpicture}
\caption{The dispace of a concurrent program and its time-reversed dispace.}
\label{fig:timereversal}
\end{figure}

\subsection{Main results and organisation of the article}

Let us present our main results and the organization of the article.
In the next section we recall categorical preliminaries used in our constructions. In Subsections~\ref{SS:InternalGroups} and ~\ref{SS:CategoryActions}, we recall the notion of internal group over a category and the category $\act$ of actions as introduced by Grandis in~\cite{marco2013homological}. We recall also the embeddings of the categories $\gp$ of groups and $\psetcat$ of pointed sets into the category $\act$. These embeddings preserve exactness of sequences when $\act$ is endowed with the structure of a homological category presented in~\cite{marco2013homological}.
In Subsection~\ref{SS:NaturalSystemsWithCompositionPairing}, we recall the notion of natural system, central in this work. These were introduced in \cite{Quillen72} and used as coefficients for cohomology of small categories in~\cite{BauesWirsching85} and monoids in~\cite{Leech85}, as well as to define homological finiteness invariants for convergent rewriting systems in~ \cite{GuiraudMalbos12advances, GuiraudMalbos10smf}.
A \emph{natural system on a category $\Cr$} with values in a category $\V$ is a functor $D : \Fact{\Cr} \fl \V$, where~$\Fact{\Cr}$ is the category of factorization of $\Cr$ whose $0$-cells are the $1$-cells of $\Cr$ and the $1$-cells correspond to factorizations of $1$-cells in $\Cr$, see~\ref{SSS:NaturalSystem}.
We denote by $\opnat{\V}$ the category of pairs $(\Cr,D)$ where $\Cr$ is a category and $D$ is a natural system on  $\Cr$ with values in $\V$.
The category of natural systems on a category $\Cr$ with values in the category $\ab$ of Abelian groups is equivalent to the category of internal groups in the category $\cat_{\Cr_0}/\Cr$ of categories over $\Cr$. In order to extend such an equivalence to natural systems with values in the category $\gp$ of groups, Porter in \cite{Porter2} considers natural systems enriched with composition pairings. Specifically, given a natural system $D : \Fact{\Cr} \fl \V$, a \emph{composition pairing} associated to $D$ consists of families
\[
\nu_{f,g} : D_f\times D_g \fl D_{fg}
\qquad
\nu_x : T \fl D_{1_x},
\]
of morphisms of $\V$ indexed by $1$-cells $f,g$ a $0$-cell $x$ in $\Cr$, satisfying coherence conditions as recalled in~\ref{SSS:NaturalSystemsCompositionPairing}. Porter showed that the category of natural systems on a category $\Cr$ with values in the category of groups and with composition pairings is equivalent to the category of internal groups in the category of categories over $\Cr$. We recall this equivalence in~\ref{S:LaxSystemToInternalGroups} and explain in \ref{S:SplitObject} that such an equivalence can equally be established for natural systems with values in the category $\psetcat$ by considering split objects in the category of categories over $\Cr$.
Finally, in Subsection~\ref{SS:NatLax} we recall the definition of lax functors and their relation to natural systems with composition pairings, again from \cite{Porter2}.

The aim of Section~\ref{S:DirectedHomotopyInternalGroup} is to relate the notion of directed homotopy of directed spaces to certain internal groups, refining this invariant of directed spaces. We recall notions from directed algebraic topology in Subsection~\ref{subsec:dihom}. In particular, we define the functor $\overrightarrow{\mathbf{P}} : \dtop \fl \cat$ which associates to a dispace $\Xr$ the \emph{trace category of $\Xr$}, whose $0$-cells are points of $X$, $1$-cells are traces of $\Xr$, \ie classes of dipaths of $\Xr$ modulo reparametrization, in which composition is given by concatenation of traces.
We are interested in the properties of the natural homotopy and natural homology functors as introduced in~\cite{Eilenberg,dubutthesis}, see~\ref{SSS:NaturalHomotopyNaturalHomology}.
The functors $\sysp{n}{\Xr}$ and $\sysh{n}{\Xr}$, for a dispace $\Xr$, are natural systems extending the homotopy and homology functors on topological spaces to directed spaces. They extend to functors 
\[
\overrightarrow{P}_{\!\!n} : \dtop \fl \opnat{\act},
\quad\text{and}\quad
\overrightarrow{H}_n : \dtop \fl \opnat{\ab},
\]
sending a dispace $\Xr$ to $(\diP{\Xr},\sysp{n}{\Xr})$ and $(\diP{\Xr},\sysh{n}{\Xr})$ respectively.
In Subsection~\ref{SS:DirectedHomotopyAsInternalGroup}, we show that the natural systems $\sysp{n}{\Xr}$ and $\sysh{n}{\Xr}$ admit composition pairings. This additional structure allows us to relate the natural systems $\sysp{n}{\Xr}$ and $\sysh{n}{\Xr}$ to internal groups or split objects in the category~$\textbf{Cat}_{X}/ \diP{\Xr}$, giving the main result of this section:
\begin{quote}
\noindent \textbf{Theorem~\ref{T:CompositionPairingOnPinX}.}
\emph{
Let $\Xr=(X,dX)$ be a dispace.
For each $n\leq 1$ (resp.~$n\geq 2$) there exists a split object $\Cr^n_\Xr$ (resp.~internal group $\Cr^n_\Xr$) in $\catxp$ such that}
\[
\sysp{n}{\Xr}_f = (\Cr^n_\Xr)_f,
\]
\emph{for all traces $f$ of $\Xr$, and this assignment is functorial in $\Xr$.}
\end{quote}
In this way, we have defined a functor $\Cr^n_- : \dtop \fl \cat$.
We explain in \ref{SSS:NaturalHomologyAsInternalGroup} that for all $n\geq 1$, the natural system $\sysh{n}{\Xr}$ is naturally equipped with a composition pairing, and thus can be interpreted as an internal abelian group $\Ar_\Xr^n$ in the category $\catxp$. Moreover, the assignment $\Ar_-^n:\dtop\fl\cat$ is functorial for all $n\geq 1$.
Finally, in Subsection~\ref{SS:FundamentalCategory} we recall the notion of fundamental category from~\cite{grandisbook} and provide a result relating it to natural homotopy.

In Section~\ref{S:TimeReversalInvariance} we examine the effect of time-reversal on dispaces and study the behaviour of natural homology and natural homotopy on dispaces with respect to time-reversal. First, we define the time-reversed, or opposite, dispace of a dispace $\Xr$ as the dispace $\Xr^\sharp=(X,dX^\sharp)$, where $dX^\sharp$ is the set of time-reversed dipaths in $X$. 
For every $n\geq 0$, we explicit an isomorphism
\[
I_n(\Xr): \Cr^n_{\Xr^\sharp} \Longrightarrow (\Cr^n_{\Xr})^o,
\]
which is natural in the $\Xr$,  showing the main result of this paper:
\begin{quote}
\noindent \textbf{Theorem~\ref{Theorem:MainTheoremA}.}
\emph{
For any $n\geq 0$, the functor $\Cr^n_{(-)}: \dtop \fl \cat$ is strongly time-reversal.
}
\end{quote}

Finally, in Subsection~\ref{subsec:relhom}, we introduce a notion of relative homotopy for dispaces, and establish a long exact sequence, as in the case of regular topological spaces, using the homological category structure on $\act$ as introduced by Grandis in~\cite{marco2013homological}:

\begin{quote}
\noindent \textbf{Theorem~\ref{T:MainTheoremB}.}
\emph{
Let $\Xr$ be a dispace and $\Ar$ be a directed subspace of $\Xr$.
There is an exact sequence in $\natsys{\diP{\Ar}}{\act}$:
\begin{align*}
\cdots 
\: \rightarrow \: \sysp{n}{\Ar} \: \rightarrow \: \sysp{n}{\Xr}\: \rightarrow \: & \sysp{n}{\Xr,\Ar} \: \overset{\partial_n}{\rightarrow} \: \sysp{n-1}{\Ar} \: \rightarrow \: \cdots \\
\cdots \: \rightarrow \: \sysp{2}{\Ar} \: \overset{v}{\rightarrow} \sysp{2}{\Xr} \: \overset{f}{\rightarrow} \: (\sysp{2}{\Xr,\Ar}, \sysp{2}{\Xr}) & \: \overset{g}{\rightarrow} \: \sysp{1}{\Ar} \overset{h}{\rightarrow} \: \sysp{1}{\Xr} \: \rightarrow \: \sysp{1}{\Xr,\Ar} \: \rightarrow \: 0.
\end{align*}
}
\end{quote}

\subsubsection*{Acknowledgement}
We wish to thank Kathryn Hess, Lisbeth Fajstrup and Timothy Porter for fruitful discussions and comments about this work. In particular, we thank Kathryn Hess and Lisbeth Fajstrup for pointing out that the natural system of directed homology is time-symmetric.

\section{Categorical preliminaries}
\label{S:CategoricalPreliminaries}

In this section we recall categorical constructions used in this article. First we recall the notion of an internal group over a category. Then, in Subsection~\ref{SS:CategoryActions} we recall the embeddings of the categories $\gp$ of groups and $\psetcat$ of pointed sets into the category $\act$ of actions, and their exactness properties as shown in~\cite{marco2013homological}.
In Subsection~\ref{SS:NaturalSystemsWithCompositionPairing}, we recall the notion of natural system on a category as well as the notion of composition pairing associated to a natural system, allowing the description of natural systems of groups in terms of internal groups over a category \cite{Porter2}. Finally, in Subsection~\ref{SS:NatLax} we recall the definition of lax functors and their relation to natural systems with composition pairings, again from \cite{Porter2}.

\subsection{Internal groups}
\label{SS:InternalGroups}

We denote by $\cat$ the category of (small) categories. For a category $\Cr$, we will denote by $\Cr_0$ its set of $0$-cells and by $\Cr_1$ its set of $1$-cells. Given a set $X$, we denote by $\catsigma$ the subcategory of $\cat$ consisting of those categories with $X$ as their set of $0$-cells, and in which we take only the functors which are the identity on $0$-cells. 
Given a category $\Br$, we denote by $\catsigmab$ the category whose objects are pairs $(\Cr, p)$, with~$\Cr$ in~$\cat_{\Br_0}$, and where $\map{p}{\Cr}{\Br}$ is a functor which is the identity on $0$-cells. A morphism in $\cat_{\Br_0}/\Br$ from $(\Cr,p)$ to $(\Cr',p')$ is a functor $f:\Cr \fl \Cr'$ such that the following diagram commutes in $\cat_{\Br_0}$
\[
\xymatrix @C=4em @R=1em {
\Cr  
  \ar[rr] ^-{f}
  \ar[dr] _-{p}
&& 
\Cr' 
  \ar[dl] ^-{p'}
\\
& 
\Br 
&
}
\]
Note that $\catsigmab$ has arbitrary limits, and that its terminal object is the pair $(\Br, id_{\Br})$. Given an object $(\Cr, p)$ in $\catsigmab$ and a $1$-cell $\map{f}{x}{y}$ of $\Br$, the \emph{fibre of $f$ in $\Cr$}, denoted by $\Cr_f$, is the pre-image of $f$ in~$\Cr$ by $p$, that is 
\[
\Cr_{f} = \set{\map{c}{x}{y}\;\;\text{in}\;\; \Cr_1 \hskip.2cm | \hskip.2cm p(c)=f}.
\]

\subsubsection{Internal groups} 
Let $\A$ be a category with finite products and denote by $T$ its terminal object. Recall from~{\cite[III.6]{MacLane98} that an \emph{(internal) group} in $\A$} is a tuple $\mathcal{G}=(G, \mu, \eta, (-)^{-1})$, where $G$ is an object of $\A$, and
\[
\map{\mu}{ G \times G}{G},
\qquad
\map{\eta}{ T}{G},
\quad\text{and}\quad
\map{(-)^{-1}}{ G}{G}
\]
are morphisms of $\A$, respectively called the \emph{multiplication}, \emph{identity}, and \emph{inverse} maps, which must satisfy the group axioms.
A \emph{morphism of internal groups} from $\mathcal{G}$ to $\mathcal{G}'$ is a morphism $f : G \fl G'$ of $\A$ that commutes with the associated multiplication and identity morphisms. The category of internal groups in~$\Ar$ and their morphisms is denoted by $\gpos{\A}$. 
The groups which additionally satisfy the commutativity condition $\mu = \mu \circ \tau$, where $\tau$ exchanges the factors of the product, constitute a full subcategory of $\gpos{\A}$ called the \emph{category of abelian groups} in $\A$, denoted by $\abo{\A}$.

\subsubsection{}
We now turn to the case of groups in the category $\catsigmab$. 
Since $(\Br, id_{\Br})$ is its terminal object, given a group $\gpo{\Cr,p}$ of $\catsigmab$, the following diagram is commutative in $\cat_{\Br_0}$: 
\[
\xymatrix@C=4em @R=1.2em{
\Br 
	\ar[rd] _-{id_{\Br}}
	\ar[rr] ^-{\eta}
&   
& 
\Cr 
	\ar[dl] ^-{p} 
\\
&  
\Br 
&
\\
}    
\]
This implies that every fibre $\Cr_{f}$ is non-empty, and that $\eta$ splits $p$ in $\cat_{\Br_0}$. Therefore each hom-set $\Cr(x,y)$ is a coproduct of the fibres:
\[
\Cr(x,y) = \coprod_{f\in \Br(x,y)} \Cr_{f},
\]
whose elements are denoted by $(c,f)$, with $c\in \Cr_f$.
The product in $\catsigmab$ of an object $(\Cr,p)$ with itself is given by pullback over $\Br$ in $\cat_{\Br_0}$, and is denoted by $(\Cr \times_{\Br} \Cr, \widetilde{p})$, where the category $\Cr \times_{\Br} \Cr$ has $\Br_0$ as its set of $0$-cells and for $x,y$ in $\Br_0$, 
\[
(\Cr \times_{\Br} \Cr)(x,y) = \set{(c,d)\in {\Cr(x,y)}^{2} \hskip.2cm | \hskip.2cm p(c) = p(d)}.
\]
The functor $\widetilde{p}$ is the identity on $0$-cells, and assigns to each pair $(c,d)$ of 1-cells in $\Cr \times_{\Br} \Cr$ their common image under $p$. The hom-sets of this product thus admit the following decomposition:
\[
(\Cr \times_{\Br} \Cr)(x,y) = \coprod_{f\in\Br(x,y)} \Cr_{f}\times \Cr_{f}.
\]
Furthermore, by definition of $\mu$, we have that the following diagram commutes in $\cat_{\Br_0}$:
\[
\xymatrix@C=4em @R=1.2em{
\Cr \times_{\Br} \Cr 
	\ar[rr] ^-{\mu}
	\ar[dr] _-{\widetilde{p}}
&   
& 
\Cr  
	\ar[dl] ^-{p}
\\
&  
\Br 
&
\\
}
\]
Thus, for all $c,d \in \Cr_{f}$, we have $ f = \widetilde{p}(c,d) = p(\mu(c,d))$, and therefore $\mu(c,d) \in \Cr_{f}$. As a consequence we obtain induced maps $\map{\mu_{f}}{\Cr_{f}\times \Cr_{f}}{\Cr_{f}}$ for each $1$-cell $f$ of $\Br$.
This endows each fibre with a group structure.

\subsubsection{The opposite group}
The \emph{opposite group} of an internal group $\gpo{\Cr,p}$ in $\catsigmab$ is the internal group $(\Cr^o, p^o)$ in $\cat_{\Br_0} / \Br^o$, for which the multiplication, identity and inverse maps, denoted respectively by $\mu^o$, $\eta^o$ and $(-)^{-1}_o$, are the induced opposite maps of $\mu$, $\eta$ and $(-)^{-1}$. Note that the fibre group $\Cr_f$ in $\Cr$ associated to a $1$-cell $f$ of $\Br$ is equal to the fibre group $\Cr^o_{f^o}$ associated to its opposite $f^0$.

\subsection{The category of actions}
\label{SS:CategoryActions}

\subsubsection{The category of actions}
Recall from \cite{marco2013homological} the definition of the category of actions of groups on pointed sets, denoted by $\act$. Objects of $\act$ are \emph{actions}, defined as pairs $(X,G)$ where $X$ is a pointed set, whose base point we shall denote by $0_X$, and $G$ is a group with identity element $1_G$, equipped with a right action of $G$ on $X$. The base point of $X$ is not assumed to be fixed by the action, and we will write
\[
G_0 = Fix_{G}(0_X) = \set{g\in G \hskip.2cm | \hskip.2cm 0_X \cdot g = 0_X}
\]
to denote the subset of $G$ fixing the base point.
A morphism in $\act$ is a pair $\map{f=(f',f'')}{(X,G)}{(Y,H)}$ where $\map{f'}{X}{Y}$ is a morphism of pointed sets, and $\map{f''}{G}{H}$ is a morphism of groups compatible with the action, in the sense that for all $g\in G$ and all $x\in X$,
\[
f'(x\cdot g ) = f'(x)\cdot f''(g).
\]

We consider $\act$ as a homological category as introduced by Grandis in {\cite[Section 6.4]{marco2013homological}}. With this structure the kernel of a morphism $\map{f}{(X,G)}{(Y,H)}$ is the inclusion
\[
(\Kernel{f'}, {f''}^{-1}(H_0))\fl (X,G).
\]
where $\Kernel{f'} = {f'}^{-1}(\set{0_Y})$. Observe that ${f''}^{-1}(H_0)$ is the subset of $G$ consisting of elements $g$ such that $ x = x'\cdot g$ for some $x,x'\in \Kernel{f'}$.
Duly the cokernel of a morphism $\map{f}{(X,G)}{(Y,H)}$ is the projection 
\[
(Y,H)\fl(Y/R, H),
\]
where $R$ is an equivalence relation on $Y$ defined by $y \equiv_{R} y' $ if and only if either $y$ or $y'$ is an element of~$f(X)$ and there exists some $h\in H$ with $y = y'\cdot h$.

\subsubsection{Embeddings of $\gp$ and $\psetcat$ in $\act$}
\label{SSS:EmbeddingInAct}
There are embeddings of the categories $\gp$ and $\psetcat$ into the category $\act$ that preserve exactness of sequences and morphisms. 
In the case of $\psetcat$, there are adjoint functors,
\[
J: \psetcat \fl \act,
\qquad
V: \act \fl \psetcat,
\]
defined by $J(X) = (X, \set{1})$ and $V(X,G) = X/G$ for all pointed sets $X$ and groups $G$ with a right action on~$X$, where $(X, \set{1})$ is the action of the trivial group on $X$, and $X/G$ is the quotient of $X$ by the $G$-orbits of the action, pointed at the class of $0_X$. 
The functor $J$ induces an equivalence of categories between $\psetcat$ and the full homological subcategory of $\act$ consisting of actions of the trivial group. This, along with the fact that $J$ preserves null morphisms, means that it preserves exactness of sequences.

On the other hand, the category $\gp$ can be realized as a retract of the category $\act$, via the functors 
\[
K: \gp \fl \act,
\qquad
R: \act \fl \gp,
\]
defined by $K(G) = (|G|, G)$ and $R(X,G)= G/\dual{G_0}$, where $(|G|, G)$ is the usual right action of $G$ on the underlying set $|G|$, pointed at $1_G$. Recall that this action is transitive. In the definition of $R$, $G/\dual{G_0}$ is the quotient of $G$ by the invariant closure in $G$ of the subgroup $G_0$ stabilizing the base point $0_X$ of $X$. 
These show that $\gp$ is a retract of $\act$ in the sense that $R\circ K = id_{\gp}$ since the action of $G$ on itself is transitive. As a consequence, a sequence of groups viewed in $\act$ is exact if and only if the sequence is exact in the usual sense.

\subsection{Natural systems with composition pairing}
\label{SS:NaturalSystemsWithCompositionPairing}

\subsubsection{Natural systems}
\label{SSS:NaturalSystem}
The \emph{category of factorizations} of a category $\Cr$, denoted by $\Fact{\Cr}$, is the category whose $0$-cells are the $1$-cells of $\Cr$, and a $1$-cell from $f$ to $f'$ is a pair $(u,v)$ of $1$-cells of $\Cr$ such that $ufv =f'$ holds in $\Cr$. Composition is given by 
\[
(u,v)(u',v') = (u'u,vv'),
\]
whenever $u'$ and $v$ are composable with $u$ and $v'$ respectively, and the identity on $\map{f}{x}{y}$ is the pair~$(1_x , 1_y)$. 
A \emph{natural system on a category $\Cr$ with values in a category $\V$} is a functor
\[
\map{D}{\Fact{\Cr}}{\V}.
\]
We will denote by $D_{f}$ (resp.~$D(u,v)$) the image of a $0$-cell $f$ (resp. $1$-cell $(u,v)$) of $\Fact\Cr$.
In most cases, we will consider natural systems with values in the category $\psetcat$ of pointed sets, the category $\gp$ of groups, the subcategory $\ab$ of abelian groups, or the category $\act$, then called \emph{natural systems of pointed sets}, \emph{of groups}, \emph{of abelian groups}, or \emph{of actions} respectively.

We denote by $\natsys{\Cr}{\V}$ the category whose objects are natural systems on $\Cr$ with values in $\V$ and in which morphisms are natural transformations between functors.
The category of natural systems with values in $\V$, denoted by $\opnat{\V}$, is defined as follows:
\begin{enumerate}[{\bf i)}]
\item its objects are the pairs $(\Cr,D)$ where $\Cr$ is a category and $D$ is a natural system on $\Cr$ with values in $\V$,
\item its morphisms are pairs
\[
(\Phi,\tau): (\Cr,D)\rightarrow (\Cr',D')
\]
\noindent consisting of a functor $\Phi : \Cr \rightarrow \Cr'$ and a natural transformation $\tau: D \rightarrow \Phi^*D'$, where the natural system $\Phi^*D': \Fact{\Cr} \rightarrow \V$ is defined by
\[
(\Phi^*D')(f)=D'(\Phi f),
\]
\noindent for any $1$-cell $f$  in $\Cr$ and $\Phi^*D'(u,v)=D'(\Phi(u),\Phi(v))$ for any $1$-cells $u$ and $v$ in $\Cr$, 
\item composition of morphisms $(\Psi,\sigma)$ with $(\Phi,\tau)$ is defined by
\[
(\Psi,\sigma) \circ (\Phi,\tau)=(\Psi \circ \Phi,(\Phi^*\sigma) \circ \tau).
\]
\end{enumerate}

\subsubsection{Natural systems and composition pairings}
\label{SSS:NaturalSystemsCompositionPairing}
Let $\V$ be a category with finite products. Given a natural system $D$ on a category $\Cr$ with values in $\V$, recall from~\cite{Porter2} that a \emph{composition pairing} associated to $D$ consists of two families of morphisms of~$\V$ 
\[
\big(\; \nu_{f,g} : D_f\times D_g \fl D_{fg}\;\big)_{f,g\in \Cr_1} \qquad \text{ and } \qquad \big(\; \nu_x : T \fl D_{1_x}\;\big)_{x\in \Cr_0},
\]
where $T$ is the terminal object in $\V$, and such that the three following coherent conditions are satisfied:
\begin{enumerate}[{\bf i)}]
\item \emph{naturality condition}:  the following diagram
\[
\xymatrix @C=4em @R=2.5em{
D_{f}\times\ D_{g}
  \ar[r] ^-{\nu_{f,g}}
  \ar[d] _-{D(u,1)\times D(1,v)}
& 
D_{fg} 
  \ar[d] ^-{D(u,v)}
\\
D_{uf}\times D_{gv} 
  \ar[r] _-{\nu_{uf,gv}}
&  
D_{ufgv}
}
\]
commutes in $\V$ for all $1$-cells $f,g,u,v$ in $\Cr_1$ such that the composites are defined.
\item \emph{The cocycle condition:} the diagram 
\[
\xymatrix@C=6em@R=2.5em{
D_{f}\times D_{g}\times D_{h}
	\ar[r] ^-{\nu_{f,g}\times id_{D_{h}}}
	\ar[d] _-{id_{D_{f}}\times \nu_{g,h}}
&  
D_{fg}\times D_{h} 
	\ar[d] ^-{\nu_{fg,h}}
\\
D_{f}\times D_{gh} 
	\ar[r] _-{\nu_{f,gh}}
& 
D_{fgh}
}
\]
commutes for all $1$-cells $f,g$ and $h$ of $\Cr$ such that the composite $fgh$ is defined,
\item \emph{The unit conditions:} the diagrams
\[
\xymatrix@C=4em@R=2.5em{
D_{f}  
& 
D_{f} \times D_{1_{y}} 
	\ar[l] _-{\nu_{f,1_{y}}}
\\
&   
D_{f}\times T
	\ar[u] _-{1_{D_{f}}\times\nu_{y}}
	\ar[ul] ^-{\cong}
}
\hskip1cm
\xymatrix@C=4em@R=2.5em{
D_{1_{x}}\times D_{f} 
	\ar[r] ^-{\nu_{1_{x},f}} 
& 
D_{f}
\\
T\times D_{f}   
	\ar[u] ^-{\nu_{x}\times 1_{D_{f}}}
	\ar[ur] _-{\cong}
& 
}    
\]
commute for every $1$-cell $\map{f}{x}{y}$ of $\Cr$.
\end{enumerate}

The category of natural systems on $\Cr$ with values in $\V$ which admit a composition pairing is the category whose objects are pairs $(D,\nu)$, with $D$ a natural system on $\Cr$ and $\nu$ a composition pairing associated to $D$. The morphims are natural transformations $\alpha : D \fl D'$ \emph{compatible with the composition pairings} $\nu$ and $\nu'$, in the sense that the following diagram commutes in $\V$
\[
\xymatrix@C=5em@R=2.5em{
D_{f}\times\ D_{g} 
	\ar[r] ^-{\nu_{f,g}}
	\ar[d] _-{\alpha_{f}\times\alpha_{g}}
& 
D_{fg}
	\ar[d] ^-{\alpha_{fg}}
\\
D'_{f}\times D'_{g} 
	\ar[r] _-{\nu'_{f,g}}
&  
D'_{fg} 
}
\]
for all composable $1$-cells $f$ and $g$ in $\Cr$.
We will denote this category of natural systems admitting a composition pairing by $\natsysnu{\Cr}{\V}$. We denote by $\opnatnu{\V}$ the subcategory of $\opnat{\V}$ consisting of natural systems with values in $\V$ which admit a composition pairing, in which we take only those morphisms $(\Phi,\tau):(\Cr,(D,\nu)) \fl (\Cr', (D',\nu'))$ such that $\tau$ is compatible with the composition pairings~$\nu$ and $\Phi^*\nu'$.

\subsubsection{Commutator condition}\label{subsubsec:compcomm}
Consider a natural system of groups $D : \Fact\Br \fl \gp$.
For all composable $1$-cells $f$ and $g$ of $\Br$, define a homomorphism $\map{\nu_{f,g}}{D_{f}\times D_{g}}{ D_{fg}}$, by setting
\[
\nu_{f,g}(d,d') =  D(f,1)(d').D(1,g)(d),
\]
for all $d\in D_f$ and $d'\in D_g$, where the right hand side is a product in $D_{fg}$. Porter proved in \cite{Porter2} that a natural system of groups $D$ on a category $\Br$ admits a composition pairing if, and only if, the condition 
\[
[D(f,1)(d'),D(1,g)(d)] = 1,
\]
holds for all $d\in D_{f}$, and $d'\in D_g$. In this case, the composition pairing is uniquely given by 
\[
\nu_{f,g}(d,d') =  D(f,1)(d').D(1,g)(d) =  D(1,g)(d).D(f,1)(d'),
\]
for all $1$-cells $f,g$ and $d\in D_f$ and $d'\in D_g$ such that the composition $fg$ is defined.
Note that as a consequence of this characterization, every natural system of abelian groups admits a composition pairing,~\cite{Porter2}. 

\subsubsection{Remarks}
\label{RemarkCompatility}
The compatibility condition for natural transformations is always satisfied in the case of natural systems of groups with composition pairings.
Indeed, if $\alpha: D\fl D'$ is a transformation of natural systems, we have 
\[
D'(1,g)(\alpha_{f}(d)) = \alpha_{fg}(D(1,g)(d)) \text{ and } D'(f,1)(\alpha_{g}(d')) = \alpha_{fg}(D(f,1)(d'))
\] 
for all $d\in D_f$ and $d'\in D_g$. Thus
\[
\alpha_{fg}(\nu_{f,g}(d,d')) =
\nu'_{f,g}(\alpha_{f}(d), \alpha_{g}(d')).
\]
We thereby deduce that~$\natsysnu{\Br}{\gp}$ (resp.~$\opnatnu{\gp}$) is a full subcategory of~$\natsys{\Br}{\gp}$ (resp.~$\opnat{\gp}$), and that the categories $\natsys{\Br}{\ab}$ and $\natsysnu{\Br}{\ab}$ are equal.

\subsubsection{Natural systems and internal groups}
\label{S:LaxSystemToInternalGroups}
Given a natural system $D : \Fact\Br \fl \gp$ with composition pairing $\nu$, we construct an internal group in the category $\mathbf{Cat}_{\Br_0}/ \Br$.
First, we construct a category $\Cr$ of~$\mathbf{Cat}_{\Br_0}$, whose hom-sets are defined as
\[
\Cr(x,y) := \coprod_{f\in\Br(x,y)} D_{f},
\]
for all $x$ and $y$ in $\Br_0$.
The $1$-cells of $\Cr$ are denoted by pairs $(c,f)$ where $c\in D_{f}$. 
For all $0$-cells $x,y,z$ of $\Br$, the $0$-composition maps 
\[
\star_0^{x,y,z} : \Cr(x,y) \times \Cr(y,z) \fl \Cr(x,z)
\]
are defined fibre by fibre using the decomposition
\[
\Cr(x,y)\times \Cr(y,z) \quad = \coprod_{f\in B(x,y), g\in B(y,z)} D_{f}\times D_{g},
\]
and the homomorphisms $\map{\nu_{f,g}}{D_{f}\times D_{g}}{D_{fg}}$, by setting
\[
(c,f)\star_0(d,g) := (\nu_{f,g}(c,d) , fg),
\]
for all $c$ in $D_f$ and $d$ in $D_g$. 
The associativity of composition $\star_0$ is a consequence of the cocycle condition, and the identity on a $0$-cell $x$ is the pair $(1_{D_{1_x^\Br}}, 1^{\Br}_x)$, where $1^{\Br}_x$ denotes the identity on $x$ in $\Br$. 

Let $\map{p}{\Cr}{\Br}$ denote the functor which is the identity on $0$-cells, and which assigns the pair $(c,f)$ to $f$. Then the pair $(\Cr,p)$ is an object of $\mathbf{Cat}_{\Br_0}/ \Br$. Now let us see that it is an internal group. The unit functor $\eta : (\Br, id_{\Br}) \fl(\Cr, p)$ is induced by the functor $\map{\eta}{\Br}{\Cr}$ defined by $\eta(f) := (1_{D_{f}} , f)$, for every $1$-cell~$f$ in $\Br$.
The multiplication map $\mu : (\Cr,p) \times (\Cr,p) \fl (\Cr,p)$ is defined by $\mu((c,f),(d,f)) := (c.d,f)$, for all $c,d$ in $D_f$ and where $c.d$ denotes the product of $c$ and $d$ in $D_f$.
The functoriality of $\mu$ is a consequence of $\nu_{f,g}$ being a homomorphism of groups.
Finally, the inverse map $(-)^{-1} : (\Cr,p) \fl (\Cr,p)$ is induced by the functor $(-)^{-1} : \Cr \fl \Cr$ defined by $(c,f)^{-1} = (c^{-1},f)$, for all $c$ in $D_f$.
This construction induces a functor 
\[
\natsysnu{\Br}{\gp} \fl \gpos{\mathbf{Cat}_{\Br_0}/\Br},
\]
which assigns an internal group in $\mathbf{Cat}_{\Br_0}/ \Br$ to each natural system of groups on $\Br$. Porter proves in~\cite{Porter2} that this functor induces an equivalence of categories
\[
\natsysnu{\Br}{\gp} \simeq \gpos{\mathbf{Cat}_{\Br_0}/\Br}.
\]

\subsubsection{Natural systems and split objects}
\label{S:SplitObject}
Given a category $\Br$, we define \emph{the category of split objects in $\catsigmab$}, denoted by $Split(\catsigmab)$, as the full subcategory of $\catsigmab$ whose objects are pairs $((\Cr,p),\epsilon)$, where $(\Cr, p)$ is an object of $\catsigmab$ and $\epsilon$ is a morphism of $\cat_{\Br_0}$ such that the following diagram commutes in $\cat_{\Br_0}$
\[
\xymatrix@C=4em @R=1.2em{
\Br 
	\ar[rd] _-{id_{\Br}}
	\ar[rr] ^-{\epsilon}
&   
& 
\Cr 
	\ar[dl] ^-{p} 
\\
&  
\Br 
&
}    
\]

Note that internal groups are split objects.
The equivalence of categories stated in \ref{S:LaxSystemToInternalGroups} from~\cite{Porter2} can be adapted to show that there is an equivalence of categories
\[
\natsysnu{\Br}{\psetcat} \simeq Split(\catsigmab).
\]

\subsection{Natural systems with composition pairings and Lax functors}\label{SS:NatLax}

In this subsection we recall from \cite{Porter2} how the notion of composition pairing is related to the structure of lax functors, and how a natural system with composition pairing can be interpreted as such.

\subsubsection{Lax functors}
We recall that given two 2-categories $\M$ and $\N$, a \emph{lax functor} from $\M$ to $\N$ is a data consisting of
\begin{enumerate}[{\bf i)}]
\item A map $\map{F}{\M_0}{\N_0}$,
\item A functor $\map{F_{x,y}}{\M(x,y)}{\N(F(x), F(y))}$ of hom-categories for all $0$-cells $x,y$ in $\M$,
\item A 2-cell $c_{f,g} : F_{x,y}(f)F_{y,z}(g) \dfl F_{x,z}(fg)$ of $\N$, for each pair of composable $1$-cells $f$ and $g$ of $\M$, where the juxtaposition on the left (resp. right) side denotes the composition $\star_0^{\N}$ (resp. $\star_0^{\M}$),
\item A 2-cell $c_{x} : 1_{F(x)} \dfl F(1_{x})$ of $\N$ for each $0$-cell $x$ of $\M$.
\end{enumerate}

These assignments must satisfy the following three conditions: 

\begin{enumerate}[{-}]
\item \emph{The naturality condition:} the assignment $(f,g)\mapsto c_{f,g}$ is natural in $(f,g)$, in the sense that $c$ is a natural transformation between functors induced by the $F_{x,y}$, namely those corresponding to the clockwise and anticlockwise composites in the following diagram:
\[
\xymatrix@C=6em@R=2.5em{
\M(x,y)\times\M(y,z)
	\ar[r] ^-{F_{x,y}\times F_{y,z}}
	\ar[d] _-{\star^{\M}_{0}}
& 
\N(F(x),F(y)) \times\N(F(y),F(z)) 
	\ar[d] ^-{\star^{\N}_{0}}
\\
\M(x,z) 
	\ar[r] _-{F_{x,z}}
&  
\N(F(x),F(z)) 
}
\]
\item \emph{The cocycle condition:} for $1$-cells $f,g$ and $h$ such that the composite $fgh$ is defined, the following diagram commutes in $\N$
\[
\xymatrix@C=5em@R=2.5em{
F(f)F(g)F(h) 
	\ar@{=>}[r] ^-{c_{f,g}F(h)}
	\ar@{=>}[d] _-{F(f) c_{g,h}}
&  
F(fg)F(h) 
	\ar@{=>}[d] ^-{c_{fg,h}}
\\
F(f)F(gh) 
	\ar@{=>}[r] _-{c_{f,gh}}
& 
F(fgh)
}
\]
\item \emph{The left and right unit conditions:} for every $1$-cell $\map{f}{x}{y}$ of $\M$ the following diagrams commute in~$\N$:
\[
\xymatrix@C=4em@R=2.5em{
F(1_{x})F(f)  
	\ar@{=>} [r] ^-{c_{1_{x},f}}
& 
F(f) = F(1_{x} f ) 
\\
1_{F(x)} F(f) =F(f) 
	\ar@{=>} [u] ^-{c_{x}F(f)}   
	\ar@{=} [ur]
& 
}      
\qquad
\xymatrix@C=4em@R=2.5em{
F(f 1_{y}) = F(f) 
& 
F(f)  F(1_{y}) 
	\ar@{=>} [l] _-{c_{f,1_{y}}}
\\
&   
F(f) =  F(f) 1_{F(y)} 
	\ar@{=>} [u] _-{F(f)c_{y}}
	\ar@{=} [ul] 
}
\]
\end{enumerate}

\begin{remark}\label{rem:laxnaturality}
The naturality of $c$ in $(f,g)$ requires that if $\alpha : f \dfl f'$ and $\beta : g \dfl g'$ are 2-cells of~$\M$, there is a commutative diagram in~$\N$:
\[
\xymatrix@C=5em@R=2.5em{
F_{x,y}(f)F_{y,z}(g) 
	\ar@{=>} [r] ^-{c_{f,g}}
	\ar@{=>} [d] _-{F_{x,y}(\alpha)F_{y,z}(\beta)}
&  
F_{x,z}(fg) 
	\ar@{=>} [d] ^-{F_{x,z}(\alpha\beta)}
\\
F_{x',y'}(f')F_{y',z'}(g) 
	\ar@{=>} [r] _-{c_{f',g'}}
& 
F_{x',z'}(f'g')
}
\]
The transformation $c$ thus makes $F$ homotopy-equivalent to a 2-functor in the sense that it provides $F$ with a weakened functorial behaviour with respect to 0-composition of 2-cells.
\end{remark}

\subsubsection{Lax systems}
Natural systems with composition pairings on a $1$-category $\Br$ with values in a cartesian category $(\V,\times, T)$, where $T$ is the terminal object, can be interpreted as certain lax functors. For that we must view both~$\Br$ and $\V$ as $2$-categories as follows. The category $\Br$ is extended into a $2$-category, denoted by $\twoB$, by adding identity $2$-cells, and the cartesian category $\V$ is suspended into a $2$-category, denoted by~$\V[1]$. The $2$-category~$\V[1]$ has only one $0$-cell, and its $1$-cells and their $0$-composites correspond to $0$-cells in~$\V$ and their products, while its $2$-cells correspond to the $1$-cells of $\V$.

We recall from \cite{Porter2} that a \emph{lax system} on a category $\Br$  with values in a cartesian category $\V$ is a lax functor from $\twoB$ with values in $\V[1]$. It is shown in \cite{Porter2} that given a lax system $(F,c)$ on $\Br$ with values in~$\V$, we can construct a natural system $UD$ by associating to each $0$-cell $f$ of $\Fact{\Br}$ the $0$-cell $D_f$ of $\V$, and to each $1$-cell $(u,v)$ of $\Fact{\Br}$ the $1$-cell 
$D_f \rightarrow D_{ufv}$ sending $d$ to $c_{uf,v}(c_{u,f}(1,d),1)$.
We define the \emph{category of lax systems on $\Br$ with values in $\V$}, denoted by $\laxsys{\Br}{\V}$, in which a morphism from $(F,c)$ to $(F', c')$ is a natural transformation $\alpha:UD\dfl UD'$ between the corresponding underlying natural systems, such that the following diagram commutes:
\[
\xymatrix@C=3em@R=2.25em{
UD_{f}\times UD_{g} 
	\ar[r] ^-{c_{f,g}}
	\ar[d] _-{\alpha_{f}\times\alpha_{g}}
& 
UD_{f} 
	\ar [d] ^-{\alpha_{fg}}
\\
UD'_{f}\times UD'_{g} 
	\ar[r] _-{c'_{f,g}}
&  
UD'_{fg} 
}
\]

\subsubsection{Lax systems of groups}
Let us expand on the notion of lax system of groups $(D,\nu)$ on a category $\Br$. The set of $0$-cells of $\Br$ is collapsed on the single $0$-cell $\ast$ of $\gp[1]$. Each $1$-cell $\map{f}{x}{y}$ of $\Br$ gives a group $D_{f}$, and composable $1$-cells $f$ and $g$ give a group homomorphism, $\map{\nu_{f,g}}{D_{f}\times D_{g}}{D_{fg}}$. Finally, since the terminal object in $\gp$ is the trivial group, we have no choice in defining the homomorphisms $\map{\nu_{x}}{1_{\ast}=T}{D_{1_x}}$ associated to objects. 
Porter established in \cite{Porter2} the following equivalence of categories
\[
\laxsys{\Br}{\gp} \simeq \natsysnu{\Br}{\gp}.
\]

\section{Directed homotopy as an internal group}
\label{S:DirectedHomotopyInternalGroup}

In Section~\ref{subsec:dihom} we recall the notion of dispace from~\cite{grandisbook} and the definition of natural homotopy and natural homology as introduced in~\cite{Eilenberg,dubutthesis}. These are natural systems extending the classical algebraic invariants to dispaces.
In Section~\ref{SS:DirectedHomotopyAsInternalGroup}, we show that these natural systems have an associated composition pairing, and relate them to certain internal groups or split objects. Finally, in Section~\ref{SS:FundamentalCategory} we recall the notion of fundamental category from~\cite{grandisbook} and relate it to natural homotopy.

\subsection[Directed homology and homotopy]{Directed homology and homotopy}\label{subsec:dihom}
In this subsection we recall the notion of dispaces from~\cite{grandisbook}, and define algebraic invariants for these spaces, natural homotopy and natural homology, as introduced in~\cite{Eilenberg,dubutthesis}.

\subsubsection{Directed spaces}
\label{rem:dirspaces}
Recall from~\cite{grandisbook} that a \emph{directed space}, or \emph{dispace}, is a pair $\Xr=(X, dX)$, where $X$ is a topological space and $dX$ is a set of paths in $X$, \ie continuous maps from~$[0,1]$ to~$X$, called \emph{directed paths}, or \emph{dipaths} for short, satisfying the three following conditions:\begin{enumerate}[{\bf i)}] 
\item Every constant path is directed,
\item $dX$ is closed under monotonic reparametrization,
\item $dX$ is closed under concatenation.
\end{enumerate}
We will denote by $f\star g$ the concatenation of dipaths $f$ and $g$, defined via monotonic reparametrization. A morphism $\map{\varphi}{(X,dX)}{(Y,dY)}$ of dispaces is a continuous function $\varphi : X \fl Y$ that preserves directed paths, \ie, for every path $\map{p}{[0,1]}{X}$ in $dX$, the path $\map{\varphi_{*}p}{[0,1]}{Y}$ belongs to $dY$. The category of dispaces is denoted $\dtop$.
An isomorphism in $\dtop$ from $(X,dX)$ to $(Y,dY)$ is a homeomorphism from $X$ to $Y$ that induces a bijection between the sets $dX$ and $dY$. 

Note that the forgetful functor $U : \dtop \fl \top$ admits left and right adjoint functors. The left adjoint functor sends a topological space $X$ to the dispace $(X,X_d)$, where $X_d$ is the set of constant directed paths. The right adjoint sends $X$ to the dispace $(X,X^{[0,1]})$, where $X^{[0,1]}$ is the set of all paths in $X$.

For a dispace $\Xr=(X, dX)$ and $x$,$y$ in $X$, we denote by $\overrightarrow{Di}(\Xr)(x,y)$ the space of dipaths $f$ in $X$ with source $x=f(0)$ and target $y=f(1)$, equipped with the compact-open topology.

\subsubsection{The trace category}
The \emph{trace space} of a dispace $\Xr$ from $x$ to $y$, denoted by~$\trace{\Xr}{x}{y}$, is the quotient of $\overrightarrow{Di}(\Xr)(x,y)$ by monotonic reparametrization, equipped with the quotient topology.
The \emph{trace} of a dipath $f$ in $\Xr$, denoted by $\tr{f}$ or $f$ if no confusion is possible, is the equivalence class of $f$ modulo monotonic reparametrization. The concatenation of dipaths of $\Xr$ is compatible with this quotient, inducing a concatenation of traces defined by $\tr{f}\star\tr{g} := \tr{f\star g}$, for all dipaths $f$ and $g$ of $\Xr$.
We denote by $\overrightarrow{\mathbf{P}} : \dtop \fl \cat$ the functor which associates to a dispace $\Xr$ the \emph{trace category of $\Xr$}, whose $0$-cells are points of $X$, $1$-cells are traces of $\Xr$, and composition is given by concatenation of traces.

\subsubsection{Dihomotopies}\label{SSS:Dihomotopies}
The \emph{directed unit interval}, denoted by $\uparrow\!\! I$, is the dispace with underlying topological space $[0,1]$ and in which dipaths are non-decreasing maps from $[0,1]$ to $[0,1]$.
The \emph{directed cylinder}  of a dispace $\Xr$, denoted by $\dcyl{\Xr}$, is the dispace 
$\left(X\times\! [0,1] , d(\dcyl{\Xr})\right)$,
where
\[
d(\dcyl{\Xr}) = \set{c = (c_1,c_2) : [0,1] \fl X\times\! [0,1] \; | \; c_1\in dX \;\text{and $c_2$ monotonic}}.
\]
Recall from~\cite{grandisbook} that an \emph{elementary dihomotopy between morphisms} $\varphi,\psi : \Xr \fl \Yr$ is a morphism of dispaces
\[
h \ : \ \dcyl{\Xr} \fl \Yr
\]
such that $h(x,0) = \varphi (x)$ and $h(x,1) = \psi (x)$ for all $x$ in $X$. Dihomotopy between morphisms is defined as the symmetric and transitive closure of elementary dihomotopies.  
In particular, given a dispace $\Xr$ and dipaths $f$ and $g$ of $\Xr$, an \emph{elementary dihomotopy of dipaths} is an elementary dihomotopy between the morphisms $f,g : \uparrow\!\! I \fl \Xr$. Two dipaths are thus dihomotopic if there exists a zig-zag of elementary dihomotopies connecting them.

\subsubsection{Trace diagrams}
\label{SSS:TraceDiagrams}
The \emph{pointed trace diagram in $\top_*$} of a dispace $\Xr$ is the functor 
\[
\map{\overrightarrow{T}_{*}(\Xr)}{{\Fact\diP{\Xr}}}{\top_*}
\]
sending a trace $\map{\tr{f}}{x}{y}$ to the pointed topological space $(\trace{\Xr}{x}{y},\tr{f})$, and a $1$-cell $(\tr{u},\tr{v})$ of $\Fact{\diP{\Xr}}$ to the continuous map 
\[
\tr{u} \star \_ \star \tr{v} : \trace{\Xr}{x}{y} \fl \trace{\Xr}{x'}{y'}
\]
which sends a trace $\tr{f}$ to $\tr{u} \star \tr{f} \star \tr{v}$.
The functor $\overrightarrow{T}_{*}(\Xr)$ extends to a functor
\[
\overrightarrow{T}_{*} : \dtop \fl \opnat{\top_*}
\]
sending a dispace $\Xr$ to the pair $(\diP{\Xr},\overrightarrow{T}_{*}(\Xr))$.
Observe that a morphism of dispaces $\varphi : \Xr \fl \Yr$ induces continuous maps 
\[
\varphi_{x,y} : \trace{\Xr}{x}{y} \fl \trace{\Yr}{\varphi(x)}{\varphi(y)}
\]
for all points $x,y$ of $X$. Thus we obtain natural transformations between the corresponding trace diagrams:
\[
\overrightarrow{\varphi}_*: \overrightarrow{T}_{*}(\Xr) \dfl \overrightarrow{T}_{*}(\Yr).
\] 

\subsubsection{Natural homotopy and natural homology}
\label{SSS:NaturalHomotopyNaturalHomology}

Recall from \cite{Eilenberg,dubutthesis} that the \emph{$1^{st}$ natural homotopy functor of $\Xr$} is the natural system denoted by $\map{\sysp{1}{\Xr}}{{\Fact\diP{\Xr}}}{\setcat}$, and defined as the composite 
\[
{\Fact\diP{\Xr}} 
\overset{\overrightarrow{T}_{*}(\Xr)}{\longrightarrow} 
\top_* 
\overset{\pi_0}{\longrightarrow}
\psetcat,
\] 
where $\pi_0$ is the $0^{th}$ homotopy functor with values in $\psetcat$.
That is, for a trace $\tr{f}$ on $\Xr$ from $x$ to $y$,
\[
\sysp{1}{\Xr}_{\tr{f}} = ( \pi_0(\trace{\Xr}{x}{y}) , [\tr{f}] ),
\]
where $[\tr{f}]$ denotes the path-connected component of $\tr{f}$ in $\trace{\Xr}{x}{y}$.
For $n\geq 2$, the \emph{$n^{th}$ natural homotopy functor of $\Xr$}, denoted by $\map{\sysp{n}{\Xr}}{{\Fact\diP{\Xr}}}{\gp}$, is defined as the composite 
\[
{\Fact\diP{\Xr}} 
\overset{\overrightarrow{T}_{*}(\Xr)}{\longrightarrow} 
\top_* 
\overset{\pi_{n-1}}{\longrightarrow}
\gp,
\] 
where $\pi_{n-1}$ is the $(n-1)^{th}$ homotopy functor. 
Note that for $n\geq 3$, the functor $\sysp{n}{\Xr}$ has values in $\ab$.
Finally, for $n=0$, we define $\map{\sysp{0}{\Xr}}{{\Fact\diP{\Xr}}}{\psetcat}$ as the functor sending a trace $\tr{f}$ to the pointed singleton $(\set{\tr{f}},\tr{f})$.

Using the inclusion functors $J : \psetcat \fl \act$ and $K : \gp \fl \act$ defined in~\ref{SSS:EmbeddingInAct}, the classical homotopy functors  can be realized as functors $\pi_n : \top_\ast \fl \act$, for all $n\geq 0$. With this interpretation, natural homotopy can be resumed by functors
\[
\sysp{n}{\Xr} : {\Fact\diP{\Xr}} \fl \act,
\]
for all $n\geq 0$.

Recall from~\cite{Eilenberg}, that for $n\geq 1$, the~\emph{$n^{th}$ natural homology functor of $\Xr$} is the functor denoted by \linebreak $\map{\sysh{n}{\Xr}}{{\Fact\diP{\Xr}}}{\ab}$, and defined as the composite 
\[
{\Fact\diP{\Xr}} 
\overset{\overrightarrow{T}(\Xr)}{\longrightarrow} 
\top
\overset{H_{n-1}}{\longrightarrow}
\ab
\] 
where $H_{n-1}$ is the $(n-1)^{th}$ singular homology functor. 

The functors $\sysp{n}{\Xr}$ and $\sysh{n}{\Xr}$, for $\Xr$ in $\dtop$, extend to functors 
\[
\overrightarrow{P}_{\!\!n} : \dtop \longrightarrow \opnat{\act},
\quad\text{and}\quad
\overrightarrow{H}_n : \dtop \longrightarrow \opnat{\ab},
\]
sending a dispace $\Xr$ to $(\diP{\Xr},\sysp{n}{\Xr})$ and $(\diP{\Xr},\sysh{n}{\Xr})$ respectively.

\begin{proposition}
\label{P:DirectedHomotopy<->Homotopy}
Given a topological space $X$, the dispace $\Xr = (X, X^{[0,1]})$ is such that for every $x$ in $X$,
\[
\sysp{n}{\Xr}_{c_x} \cong \pi_n(X,x),
\]
where $c_x$ denotes the trace of the constant dipath equal to $x$.
\end{proposition}
\begin{proof}
Recall that for any $x\in X$, the loop space $\Omega (X,x)$ is the set of all continuous paths $p:\mathbb{S}^1 \fl X$ given the compact-open topology, and is thus homeomorphic to $\overrightarrow{Di}(\Xr)(x,x)$. As a consequence of Eckmann-Hilton duality, for any topological space $X$ and any $n\geq 1$,
\[
\pi_n(X,x) \cong \pi_{n-1}(\Omega (X,x)).
\]
The quotient of $\Omega (X,x)$ by monotonic reparametrization is the space $\trace{\Xr}{x}{x}$, and since paths in the same reparametrization class are homotopic, we have $\sysp{n}{\Xr}_{c_x} \cong \pi_n(X,x)$.
\end{proof}

As a consequence, given a dispace $\Xr = (X, X^{[0,1]})$, if $X$ is $n$-connected, then for every $x\in X$, the space $\trace{\Xr}{x}{x}$ is also $(n-1)$-connected. Applying the Hurewicz theorem, Proposition~\ref{P:DirectedHomotopy<->Homotopy} yields the following result.
\begin{corollary}
For $n\geq 1$ an $(n-1)$-connected topological space $X$, the dispace $\Xr = (X, X^{[0,1]})$ is such that for every $x$ in $X$
\[
\sysh{i}{\Xr}_{c_x} \cong H_i(X),
\]
for all $i\leq n$.
\end{corollary}

\subsection{Directed homotopy as an internal group or a split object}
\label{SS:DirectedHomotopyAsInternalGroup}

In this subsection we show that for any dispace $\Xr$, the natural systems $\sysp{n}{\Xr}$ and $\sysh{n}{\Xr}$ admit composition pairings.
We treat the case $\sysp{1}{\Xr}$ in Lemma~\ref{L:CompositionPairingOnPinX_n=1} separately from the the case $\sysp{n}{\Xr}$ for $n\geq 2$ in Lemma~\ref{L:CompositionPairingOnPinX_n>=1}. Finally, using the equivalence of categories stated in~\ref{S:LaxSystemToInternalGroups}, we describe the natural homotopy functor $\sysp{n}{\Xr}$ as split objects, or internal group when $n\geq 2$, in the category $\textbf{Cat}_{X}/ \diP{\Xr}$. We also treat the case of the natural homology functors $\sysh{n}{\Xr}$, for $n\geq 1$, which we describe as intrernal abelian groups in the category $\textbf{Cat}_{X}/ \diP{\Xr}$.

\begin{lemma}
\label{L:CompositionPairingOnPinX_n=1}
The natural system of pointed sets $\sysp{1}{\Xr}$ admits a composition pairing $\nu$ given,
for all composable traces $\map{f}{x}{y},\map{g}{y}{z}$ of $\Xr$, by
\[
\nu_{f,g}([f'],[g']) = [f'\star g']
\]
for any $[f']$ in $\pi_{0}(\trace{\Xr}{x}{y},f)$ and $[g']$ in $\pi_{0}(\trace{\Xr}{y}{z},g)$.
\end{lemma}
\begin{proof}
Observe that the maps $\nu_{x}:\{\ast\}\fl \sysp{1}{\Xr}$ for $x$ in $X$ are uniquely determined since the singleton is the initial object in $\psetcat$.
For composable traces $f$ and $g$ of $\Xr$, the maps $\nu_{f,g}$ are well defined and are morphisms of $\psetcat$.
Thus, we only have to check the cocycle, unit, and naturality conditions.
The cocycle condition is a consequence of the fact that the composition is associative. The right unit condition is verified, since for $\map{f}{x}{y}$, the following diagram
\[
\xymatrix@C=4em@R=3em{
\sysp{1}{X}_f  
& 
\sysp{1}{X}_f \times  \sysp{1}{X}_{1_{y}}
	\ar[l] _-{\nu_{f,1_{y}}} 
\\
&
\sysp{1}{X}_f \times \set{*} 
	\ar[u] _-{id_{\sysp{1}{X}_f}\times\nu_{y}}
	\ar[ul] ^-{\cong}
}
\]
commutes. Indeed, if $c_{y}$ denotes the constant path equal to $y$, we have 
$[f] = [f\star c_{y}] = [f\circ \nu_y(\ast)]$, since~$[c_y]$ is the pointed element of $\sysp{1}{X}_{1_{y}}$. The left unit condition is similarly verified.
Finally, the naturality condition follows from the associativity of concatenation of traces. Indeed, the equality 
\[
[(u\star f)\star(g\star v)] = [u\star (f\star g) \star v]
\]
holds for any traces $u,v,f,g$ of $\Xr$ such that the composites are defined.
\end{proof}

\begin{lemma}
\label{L:CompositionPairingOnPinX_n>=1}
For every $n\geq 2$, the natural system of groups $\sysp{n}{\Xr}$ admits a composition pairing~$\nu$ defined by
\[
\nu_{f,g}(\sigma,\tau) =  \sigma \star \tau,
\]
for all composable traces $\map{f}{x}{y}$ and $\map{g}{y}{z}$ of $\Xr$ and homotopy classes $\sigma$ in $\pi_{n-1}(\trace{\Xr}{x}{y},f)$ and $\tau$ in $\pi_{n-1}(\trace{\Xr}{y}{z},g)$, where $\sigma \star \tau$ denotes the homotopy class in $\trace{\Xr}{x}{z}$ of the map $t\mapsto \sigma(t) \star \tau(t)$. 
\end{lemma}
\begin{proof}
First observe that the maps $\nu_{x}$, for $x$ in $X$, are uniquely determined since the trivial group is the initial object in $\gp$.
Let us prove that $\sysp{n}{\Xr}$ verifies the commutator condition recalled in \ref{subsubsec:compcomm}.
Given composable $1$-cells $f$ and $g$ of $\diP{\Xr}$, the $1$-cell $(1,g)$ of $\Fact\diP{\Xr}$ induces a map 
\[
\sysp{n}{\Xr}(1,g): \pi_{n-1}(\trace{\Xr}{x}{y},f) \fl \pi_{n-1}(\trace{\Xr}{x}{z},f\star g)
\]
sending a class $\sigma$ in $\pi_{n-1}(\trace{\Xr}{x}{y},f)$ to the homotopy class of the map $t\mapsto \sigma(t) \star g$, denoted by $\sigma \star g$. We obtain a similar homomorphism from the $1$-cell $(f,1)$, sending $\tau$ in $\pi_{n-1}(\trace{\Xr}{y}{z},g)$ to the homotopy class of the map $t\mapsto f \star \tau(t)$, denoted by $f\star\tau$.

Let $\sigma, \sigma'$ in $\pi_{n-1}(\trace{\Xr}{x}{y},f)$ and $\tau, \tau'$ in $\pi_{n-1}(\trace{\Xr}{y}{z},g)$. 
The following exchange relation
\[
(\sigma \star \tau)\cdot(\sigma' \star \tau') = (\sigma \cdot \sigma')\star(\tau \cdot\tau'),
\]
where $\cdot$ denotes the product of homotopy classes in homotopy groups, holds in $\pi_{n-1}(\trace{\Xr}{x}{z},f\star g)$.
Using this relation, we have
\[
(\sigma \star g)\cdot(f\star \tau) = (\sigma \cdot f)\star(g\cdot\tau) =\sigma\star\tau= (f\cdot\sigma )\star(\tau\cdot g) = (f\star \tau) \cdot (\sigma \star g).
\]
for all $\sigma$ in $\pi_{n-1}(\trace{\Xr}{x}{y},f)$ and $\tau$ in $\pi_{n-1}(\trace{\Xr}{y}{z},g)$. We conclude via the commutator condition that~$\sysp{n}{\Xr}$ admits a composition pairing, given by $\nu_{f,g}(\sigma,\tau) =  \sigma \star \tau$.
\end{proof}

\begin{theorem}
\label{T:CompositionPairingOnPinX}
Let $\Xr=(X,dX)$ be a dispace.
For each $n\leq 1$ (resp.~$n\geq 2$) there exists a split object $\Cr^n_\Xr$ (resp.~internal group $\Cr^n_\Xr$) in $\catxp$ such that
\[
\sysp{n}{\Xr}_f = (\Cr^n_\Xr)_f,
\]
for all traces $f$ of $\Xr$, and this assignment is functorial in $\Xr$.
\end{theorem}
\begin{proof}
Using the equivalences of categories recalled in \ref{S:SplitObject} (resp.~in~\ref{S:LaxSystemToInternalGroups}), and Lemma~\ref{L:CompositionPairingOnPinX_n=1} (resp.~Lemma~\ref{T:CompositionPairingOnPinX})  we obtain a split object $\Cr^1_\Xr$ (resp.~an internal group $\Cr^n_\Xr$) in $\textbf{Cat}_{X}/ \diP{\Xr}$ associated to $\sysp{1}{\Xr}$ (resp.~$\sysp{n}{\Xr}$ for $n\geq 2$).
Let us prove that this assignment defines a functor
\[
\Cr^n_- : \dtop \fl \cat.
\]
Any morphism $\varphi:\Xr\fl\Yr$ of dispaces induces continuous maps $\varphi_{x,y}:\trace{\Xr}{x}{y}\fl\trace{\Yr}{\varphi(x)}{\varphi(y)}$ for all points $x,y$ in $X$ such that $\trace{\Xr}{x}{y}\neq\emptyset$. 
We define a functor $\Cr^n_\varphi:\Cr^n_\Xr\fl\Cr^n_\Yr$ on a $0$-cell $x$ and a $1$-cell $(\sigma,f)$ of $\Cr^n_\Xr$ by setting $\Cr^n_\varphi(x)=\varphi(x)$, and
\[
\Cr^n_\varphi(\sigma,f)=(\pi_{n-1}(\varphi_{x,y})(\sigma), \diP{\varphi}(f)). 
\]
Functoriality follows from that of $\pi_{n-1}$ and $\overrightarrow{\mathbf{P}}$.
\end{proof}

\subsubsection{Natural homotopy as a split object or internal group}
Let us describe the categories $\Cr^n_\Xr$ for~$n\geq 0$. The $0$-cells of $\Cr^n_\Xr$ are the points of $X$, and the set of $1$-cells of $\Cr^n_{\Xr}$ with source $x$ and target $y$ is given by
\[
\Cr^{n}_{\Xr}(x,y) = \coprod_{f\in\diP{\Xr}(x,y)} \sysp{n}{\Xr}_f.
\]
The projection $p$ onto the second factor extends the category $\Cr^n_\Xr$ into an object of $\cat_X/\diP{\Xr}$.

For~$n\leq 1$, the functor $p$ is split by $\epsilon_n : \diP{X} \fl \Cr^n_\Xr$ defined on any trace $f$ on $\Xr$ by $\epsilon_n(f) = ([f], f)$. 
Note that for any trace $f$, $\sysp{0}{\Xr}_f = \set{\,[f]\,}$, hence $\epsilon_0(\diP{X}) = \Cr^0_\Xr$. The composition is defined by
\[
([f'],f)([g'],g) = ([f'\star g'] , f\star g),
\]
for all $[f'] \in \sysp{n}{\Xr}_f$ and $[g'] \in \sysp{n}{\Xr}_g$. Note that $\Cr_\Xr^0$ is isomorphic to $\diP{\Xr}$.

For~$n\geq 2$, the functor $p$ is split by the identity map $\eta : \diP{X} \fl \Cr^n_\Xr$ defined by $\eta(f) = (1_{D_f},f)$, where~$1_{D_f}$ is the homotopy class of the constant loop equal to $f$. The inverse map is given by the inverse in each homotopy group, that is $(\sigma,f)^{-1} = (\sigma^{-1},f)$. Recall that the product in $\cat_X/\diP{\Xr}$ is the fibred product over $\diP{\Xr}$, so we can use the internal multiplication in each homotopy group to define the multiplication map $\mu$ by setting $\mu((\sigma,f),(\sigma',f)) = (\sigma\cdot\sigma',f)$.
The composition of $(\sigma,f)$ and $(\tau, g)$, for homotopy classes $\sigma$ and $\tau$ above $f$ and $g$ respectively, is given by
\[
(\sigma,f)\star_0 (\tau,g) = (\nu_{f,g}(\sigma,\tau), f\star g) = (\sigma\star\tau, f\star g).
\]

\subsubsection{Natural homology as internal group}
\label{SSS:NaturalHomologyAsInternalGroup}
Recall from Remark~\ref{RemarkCompatility} that as a consequence of the commutation condition and the triviality of the compatibility criterion for natural transformations, the categories $\natsys{\diP{\Xr}}{\ab}$ and $\natsysnu{\diP{\Xr}}{\ab}$ coincide. For all $n\geq 1$, the natural system~$\sysh{n}{\Xr}$ is thus equipped with a composition pairing, and via the equivalence
\[
\abo{\catxp}\cong \natsysnu{\diP{\Xr}}{\ab} 
\]
we obtain an internal abelian group $\Ar^n_\Xr$ in the category~$\catxp$.
Moreover, using similar arguments as in the proof of Proposition~\ref{T:CompositionPairingOnPinX}, one proves that the assignment $\Ar_-^n:\dtop\fl\cat$ is functorial for all~$n\geq 1$.

\subsection{Fundamental category of a dispace}
\label{SS:FundamentalCategory}

\subsubsection{Fundamental category}
The \emph{fundamental category} of a dispace $\Xr$, denoted by $\fcat{\Xr}$, is the homotopy category of $\diP{\Xr}$ when interpreted as a $2$-category. Explicitly, the trace category $\diP{\Xr}$ can be extended into a $(2,1)$-category by adding $2$-cells corresponding to dihomotopies of traces.
The fundamental category is the quotient of this $(2,1)$-category by the congruence generated by these $2$-cells. We refer the reader to \cite{gouhha,grandisbook} for a fuller treatment of fundamental categories of dispaces.
This assignment defines a functor
\[
\overrightarrow{\boldsymbol\Pi} : \dtop \fl \cat.
\]
Given a dispace $\Xr$, consider the quotient functor
$\pi : \diP{\Xr} \fl \fcat{\Xr}$,  which is the identity on $0$-cells and which associates a trace $f$ to its class $[f]$ modulo path-connectedness. Similarly to~{\cite[Theorem 1]{2017arXiv170905702G}}, we have the following result:
\begin{proposition}
Given a dispace $\Xr$, suppose that there exists a functorial section $s$ of the functor $\pi : \diP{\Xr} \fl \fcat{\Xr}$. Then the natural system $\sysp{n}{\Xr}$ is trivial for all $n\geq 2$.
\end{proposition}
\begin{proof}
We show that each trace space is contractible. Let $\overrightarrow{t}(\Xr)$ (resp.~$\overrightarrow{t}(\Xr)\times\![0,1]$) denote the natural system of topological spaces on $\fcat{\Xr}$ which associates the space $[f]\subseteq \trace{\Xr}{x}{y}$ (resp.~$[f] \times [0,1]$) to each class $[f] : x \fl y$. For a dipath $g$ in $[f]$, denote by $g_{|[s,r]}$ the restriction of $g$ to the interval $[s,r]\subseteq[0,1]$. Now we define a natural transformation $H : \overrightarrow{t}(\Xr)\times[0,1] \dfl \overrightarrow{t}(\Xr) $ such that the component $H_{[f]}$ sends a pair $(g,s)\in [f]\times[0,1]$ to the dipath

\[ 
H_{[f]}(g,s)(t) =  \begin{cases} 
      g(t)  & t\in [0,\frac{s}{2}], \\
      s([g_{|[\frac{s}{2},1 - \frac{s}{2}]}]) & t\in[\frac{s}{2}, 1 - \frac{s}{2}], \\
      g(t) & t \in [1 - \frac{t}{2}, 1] .
   \end{cases}
\]
Then $H_{[f]}(g,-)$ is a homotopy from $g$ to $s([f])$ for every $g$ in $[f]$. Thus every connected component of every trace space of $\Xr$ is contractible.
\end{proof}

\subsubsection{Remarks}
Recall that the homotopy groups $\pi_n(X,x)$ and $\pi_n(X,y)$ of a topological space $X$ are isomorphic for any path-connected points $x$ and $y$ of $X$. In the definition of natural homotopy we consider the homotopy groups of trace spaces $\trace{\Xr}{x}{y}$ based at each trace $f$.
However, choosing a single base-point in each connected component of each trace space of a dispace $\Xr$ requires a section as described above. Furthermore, for such a section to give rise to a natural system, it must be functorial. In this case, the only non-trivial homotopy functor is $\sysp{1}{\Xr}$, and this homotopic information is provided by $\fcat{\Xr}$: the hom-set $\fcat{\Xr}(x,y)$ is equal to $\pi_0(\trace{\Xr}{x}{y})$.

Finally, note that natural homology decomposes, for any trace $f: x \fl y$ on $\Xr$, into
\[
\sysh{n}{\Xr}_{f} \: \cong \: \bigoplus_{[f] \in \fcat{\Xr}(x,y)} H_{n-1}([f])
\]
where $H_{n-1}([f])$ is the $(n-1)^{th}$ singular homology of the connected space $[f] \subset \trace{\Xr}{x}{y}$.

\section{Time-reversal invariance}
\label{S:TimeReversalInvariance}

In this section we study properties on dispaces that are invariant by time-reversal. First, we define the notion of time-reversed dispace and study the behavior of natural homology and natural homotopy with respect to this reversal. Finally, in Subsection~\ref{subsec:relhom}, we introduce a notion of relative homotopy for dispaces, and establish a long exact sequence, as in the case of regular topological spaces, using the homological category structure on $\act$ as introduced by Grandis in~\cite{marco2013homological}.

\subsection{Time-reversal in dispaces}

\subsubsection{Time-reversed dispaces}
Given a dispace $\Xr=(X,dX)$, for any dipath $f$ in $dX$, we denote by $f^\sharp$ the dipath defined by
\[
f^\sharp(t) = f(1-t),
\]
for all $t$ in $[0,1]$.
We define its \emph{time-reversed dispace}, or \emph{opposite dispace}, as the dispace $\Xr^\sharp=(X,dX^\sharp)$ where $dX^\sharp$ is defined by
\[
dX^\sharp = \set{f^\sharp \; | \; f \in dX }.
\]
Note that $dX^\sharp$ is easily verified to be a set of directed paths according to the conditions listed in~\ref{rem:dirspaces}.
This defines a functor $(-)^\sharp: \dtop \fl \dtop$, sending a dispace $\Xr$ to its opposite. Notice that if $\phi:\Xr\fl\Yr$ is a morphism of dispaces, this functor leaves the continuous map $\phi:X\fl Y$ unchanged, since $(\phi_\ast f)^\sharp = \phi_\ast(f^\sharp)$.

\subsubsection{Reversal properties}
A dispace $\Xr$ is called \emph{time-symmetric} if the dispaces $\Xr$ and $\Xr^\sharp$ are isomorphic. 
In that case, by functoriality of $\overrightarrow{\mathbf{P}}$ and $\Cr^n_{-}$, there exist covariant isomorphisms
\[
\diP{\Xr} \overset{\sim}{\longrightarrow} \diP{\Xr^\sharp},
\qquad
\sysp{n}{\Xr} \overset{\sim}{\longrightarrow} \sysp{n}{\Xr^\sharp},
\qquad
\text{and} 
\qquad
\Cr_{\Xr}^n \overset{\sim}{\longrightarrow} \Cr_{\Xr^\sharp}^n.
\]

A dispace $\Xr=(X,dX)$ is called \emph{time-contractible} when $dX = dX^\sharp$. In that case any dipath is reversible, that is $f\in dX$ implies $f^\sharp \in dX$.

Note that for a dispace $\Xr=(X,dX)$, $dX = X^{[0,1]}$ implies that $\Xr$ is time-contractible but the converse is not true in general. Thus the directed homotopy of a time-contractible dispace $\Xr$ does not necessarily coincide with the homotopy of its underlying space $X$ as shown in Proposition~\ref{P:DirectedHomotopy<->Homotopy}.

A functor $F : \dtop \rightarrow \cat$ is \emph{time-reversal} if the following diagram 
\[
\xymatrix{
\dtop 
  \ar[r] ^-{F}
  \ar[d] _{(-)^{\sharp}}
&
\cat
  \ar[d] ^{(-)^{o}}
\\
\dtop
  \ar[r] _-{F}
&
\cat
}
\]
commutes up to isomorphism.
Such a functor is \emph{strongly time-reversal} if there exists a natural isomorphism $F((-)^\sharp) \dfl F(-)^o$.
A functor $F : \dtop \rightarrow \V$ is \emph{time-symmetric with respect to a category $\V$} if the following diagram 
\[
\xymatrix{
\dtop 
  \ar[r] ^-{F}
  \ar[d] _{(-)^{\sharp}}
&
\V
  \ar[d] ^{\rotatebox{90}{=}}
\\
\dtop
  \ar[r] _-{F}
&
\V
}
\]
commutes up to isomorphism. Such a functor is \emph{strongly time-symmetric with respect to $\V$} if there exists a natural isomorphism $F((-)^\sharp)\dfl F$.

\subsubsection{Time-symmetry of directed homology and homotopy}
\label{SSS:TimeSymmetricDirectedHomotopyHomology}
For any dispace~$\Xr$ the equalities
\[
\diP{\Xr^\sharp} = \diP{\Xr}^o \qquad \text{and} \qquad \fcat{\Xr^\sharp} = \fcat{\Xr}^o
\]
holds in $\cat$, hence the functors $\overrightarrow{\mathbf{P}}$ and $\overrightarrow{\boldsymbol\Pi}$ are strongly time-reversal.
The functor which sends a dispace $\Xr$ to $\Fact{\diP{\Xr}}$ is strongly time-symmetric with respect to $\cat$. Indeed, the isomorphism of categories 
\[
\Fact^\sharp : \Fact{\diP{\Xr}} \fl \Fact({\diP{\Xr^\sharp}})
\]
sending a trace $f$ to its opposite $f^\sharp$ and a $1$-cell of $(u,v)$ in $\Fact{\diP{\Xr}}$ to the $1$-cell $(v^\sharp, u^\sharp)$, is the component at $\Xr$ of a natural isomorphism.
By extension, we show that $\overrightarrow{P}_{\! n}$ and $\overrightarrow{H}_n$ are strongly time-symmetric with respect to $\opnat{\act}$. 

For $n\geq 0$, we compare the functors $\sysp{n}{\Xr}$ and $\sysp{n}{\Xr^\sharp}$ in $\natsys{\diP{\Xr}}{\act}$ by precomposing the latter with the isomorphism~$\Fact^\sharp$. 
Observe that, for all points $x,y$ in $X$, we have homeomorphisms
\[
\alpha_{x,y}: \trace{\Xr}{x}{y} \fl \trace{\Xr^\sharp}{y}{x}
\]
sending a trace $f$ to its opposite $f^\sharp$. These induce group isomorphisms 
$\sysp{n}{\Xr}_f \overset{\sim}{\longrightarrow} \sysp{n}{\Xr^\sharp}_{f^\sharp}$ for all $n \geq 2$. 
By definition, $(\Fact^\sharp)^\ast \sysp{n}{\Xr^\sharp}_f = \sysp{n}{\Xr^\sharp}_{f^\sharp}$, so we get components of a natural isomorphism
\begin{align*}
\alpha_{f}: \sysp{n}{\Xr}_f &\longrightarrow (\Fact^\sharp)^\ast \sysp{n}{\Xr^\sharp}_f \\
[\sigma] = [(s,t) \mapsto \sigma_s (t)] &\longmapsto [(s,t) \mapsto \sigma_s(1-t)] =: [\sigma^\sharp],
\end{align*}
where $s$ is the parameter for the loop in the trace space, and $t$ is the parameter for the dipath $\sigma_s$. 
Thus the pair $(\Fact^\sharp, \alpha)$ is an isomorphism in the category $opNat(\gp)$. Such an isomorphism can similarly be established for natural homotopy in the case $n=1$. The functor $F^\sharp$ and the isomorphisms are components at $\Xr$ of natural isomorphisms, hence $\overrightarrow{P}_{\!\!n}$ is strongly time-symmetric with respect to $\opnat{\act}$ for all $n\geq 1$. 

A corresponding isomorphism for natural homology,
$\sysh{n}{\Xr} \cong \sysh{n}{\Xr^\sharp}$,
can be similarly established in $\opnat{\ab}$ using the functor $\Fact^\sharp$ and the homeomorphisms $\alpha_{x,y}$, showing that $\overrightarrow{H}_n$ is strongly time-symmetric with respect to $\opnat{\ab}$ for all $n\geq 1$. 

\subsection{Time-reversibility of natural homotopy}

The covariant isomorphism $(\diP{\Xr},\sysp{n}{\Xr})\cong (\diP{\Xr^\sharp},\sysp{n}{\Xr^\sharp})$ in $\opnat{\V}$ is due to a loss of information concerning composition in the base category. Indeed, the passage to the opposite category is not witnessed by in the factorization category, as illustrated by the isomorphism of categories $\Fact{ ^\sharp}$ defined above. However, composition pairings allow us to witness the passage to the dual category $\diP{\Xr}^o = \diP{\Xr^\sharp}$ via time-reversal as a passage to the dual category of $\Cr^{n}_\Xr$.

Following Theorem~\ref{T:CompositionPairingOnPinX}, the category $\Cr^{n}_\Xr$ with the projection $p:\Cr^{n}_\Xr \fl \diP{\Xr}$ onto the second factor is an internal group in $\catxp$. On the other hand, the category $\Cr^{n}_{\Xr^\sharp}$ obtained from the natural system $\sysp{n}{\Xr^\sharp}$ via the construction given in~\ref{SS:DirectedHomotopyAsInternalGroup} has $0$-cells $x\in X$, while $1$-cells are of the form $\map{(\sigma^\sharp,f^\sharp)}{y}{x}$ where $\sigma^\sharp \in \sysp{n}{\Xr^\sharp}_{f^\sharp}$ and $\map{f^\sharp}{y}{x}$ is a trace in $\Xr^\sharp$. Composition is given by 
\[
(\tau^\sharp,g^\sharp)\star^{\Cr^{n}_{\Xr^\sharp}}_0(\sigma^\sharp, f^\sharp) = (\tau^\sharp\star\sigma^\sharp, g^\sharp\star f^\sharp).
\]
We denote the associated projection by $p^\sharp$.
We define for $n\geq 2$ 
\[
I_n(\Xr) : \Cr^{n}_{\Xr^\sharp} \fl (\Cr^{n}_\Xr)^o,
\]
the isomorphism of categories which is the identity on $0$-cells, and which sends a $1$-cell 
$(\sigma^\sharp,f^\sharp)$ of $\Cr^{n}_{\Xr^\sharp}$ to $(\sigma , f)^o$.
The functoriality of~$I_n(\Xr)$ follows from the equality
\[
(\tau^\sharp ,g^\sharp)\star^{\Cr^{n}_{\Xr^\sharp}}_0 (\sigma^\sharp, f^\sharp) = (\tau^\sharp \star \sigma^\sharp , g^\sharp\star f^\sharp) = ((\sigma\star\tau)^\sharp , (f\star g)^\sharp).
\]
The opposite group $(\Cr^n_\Xr)^o$ can be interpreted as an internal group in $\cat_X / \diP{\Xr^\sharp}$ by composing the projection $p^o: (\Cr^n_\Xr)^o \fl \diP{\Xr}^o$ with the canonical isomorphism $\diP{\Xr}^o \simeq \diP{\Xr^\sharp}$. We denote by $\widetilde{p}^o$ this composition. 
Then the following diagram commutes 
\[
\xymatrix@C=1em @R=1.5em{
\Cr^n_{\Xr^\sharp}
  \ar[rr] ^{I_n(\Xr)}
  \ar[dr] _-{p^\sharp}
&
&
(\Cr^n_\Xr)^o
  \ar[dl] ^-{\widetilde{p}^o}
\\
& \diP{\Xr^\sharp} &
}
\]
We thereby deduce that $I_n(\Xr)$ is a morphism of $\cat_X / \diP{\Xr^\sharp}$. 
Furthermore, it is a group isomorphism, since the fibre groups above a $1$-cell $f^\sharp$ of $\diP{\Xr^\sharp}$ are isomorphic:
\[
(\Cr^n_\Xr)^o_{f^\sharp} = (\Cr^n_\Xr)_f = \sysp{n}{\Xr}_f \cong \sysp{n}{\Xr^\sharp}_{f^\sharp} = (\Cr^n_{\Xr^\sharp})_{f^\sharp}.
\]
An isomophism $\Cr^{1}_{\Xr^\sharp} \cong (\Cr^{1}_X)^o$ can similarly be established in the category $Split(\cat_X/\diP{\Xr^\sharp})$. We have thus proved the following result.

 \begin{proposition}
 \label{prop:timedual1}
 Given a dispace $\Xr = (X,dX)$, $\Cr^{n}_{\Xr^\sharp}$ and $(\Cr^{n}_{\Xr})^o$ are isomorphic in $\gpos{\cat_X/\diP{\Xr^\sharp}}$ for all $n\geq 2$, and in $Split(\cat_X/\diP{\Xr^\sharp})$ for $n=1$.
In particular, the functors $\Cr^n_{-}$ are time-symmetric for all $n\geq 1$.
 \end{proposition}

\subsubsection{}
For any $n\geq 0$, the functors $I_n(\Xr)$ give components of a natural transformation. Indeed, by precomposing (resp.~composing) the functor $\Cr^n_-$ with $(\cdot)^\sharp$ (resp.~$(\cdot)^o$), any  morphism $\phi:\Xr\fl\Yr$ of dispaces yields a commuting diagram
\[
\xymatrix@C=5em{
\Cr^n_{\Xr^\sharp} 
	\ar[d] _-{\Cr^n_{\phi^\sharp}}
	\ar[r] ^-{I_n(\Xr)}
& 
(\Cr^n_\Xr)^o
	\ar[d] ^-{(\Cr^n_{\phi})^o}
 \\
\Cr^n_{\Yr^\sharp} 
	\ar[r] _-{I_n(\Yr)}
& (\Cr^n_\Yr)^o
}
\]
in $\cat$. Furthermore, as shown above, these components are all isomorphisms, that is 
there exists a natural isomorphism
\[
I_n: \Cr^n_{(-)^\sharp} \Longrightarrow (\Cr^n_{(-)})^o.
\]
We have thus proved the following result:

\begin{theorem}
\label{Theorem:MainTheoremA}
For any $n\geq 0$, the functor $\Cr^n_{(-)}: \dtop \fl \cat$ is strongly time-reversal.
\end{theorem}

A consequence of Theorem~\ref{Theorem:MainTheoremA} is that for any dispace $\Xr$, the category $\Cr^{n}_{\Xr}$ is dual to the category $\Cr^{n}_{\Xr^\sharp}$. It can similarly be shown that the functor $\Ar_{(-)}^n$ associated to natural homology is strongly time-reversal.
In the particular case of a time-symmetric spaces $\Xr$, the category $\Cr_\Xr^n$ is self-dual, \ie there exists a covariant isomorphism of categories 
\[
\Cr_\Xr^n \cong (\Cr_\Xr^n)^o.
\] 

\subsubsection{Time-reversibility with respect to opNat}
The time-reversibility of a functor with values in $\cat$ is expressed via duality of categories. However, given some category $\V$, we can define a notion of time-reversal with respect to $\opnat{\V}$ which is compatible with the interpretation of natural systems with composition pairings as categories when $\V = \act$.
Consider the functor
\[
(-)^\flat: \opnat{\V} \rightarrow \opnat{\V}
\]
which sends a pair $(\Cr, D)$ to the pair $(\Cr^o, (\Fact{^o})^\ast D)$, where $\Fact{^o} : \Fact{(\Cr^o)} \fl \Fact{\Cr}$ is the covariant functor sending a $0$-cell $f^o$ of $\Fact{\Cr^o}$ to $f$, and a $1$-cell $(v^o,u^o)$ to $(u,v)$. To a morphism
\[
(\Phi, \alpha):(\Cr, D) \rightarrow (\Cr',D')
\]
of $\opnat{\V}$, the functor $(-)^\flat$ associates the morphism $(\Phi^o, \alpha^o)$, where $\Phi^o$ is the opposite functor $\Cr^o\rightarrow(\Cr')^o$, and where the component $\alpha^o_{f^o}$ at $f^o$ a $1$-cell of $\Cr^o$ is the component $\alpha_f$ of $\alpha$ at $f$.

Then for $F$ a functor $\dtop\rightarrow\opnat{\V}$, we say that $F$ is \emph{time-reversal with respect to $\opnat{\V}$} if the following diagram
\[
\xymatrix{
\dtop 
  \ar[r] ^-{F}
  \ar[d] _-{(-)^\sharp}
& 
\opnat{\V}
  \ar[d] ^-{(-)^\flat} 
\\
\dtop 
  \ar[r] _-{F}
& 
\opnat{\V}
}
\] 
commutes up to isomorphisms of the form $(id, \alpha)$. Explicitly, this means that if $F(\Xr) = (\Cr,D)$, then $F(\Xr^\sharp) = (\Cr^o,D')$ with~$(\Fact{^o})^\ast D$ naturally isomorphic to $D'$.

Given $F$ a functor $\dtop\rightarrow\opnatnu{\act}$, we can extend to the following diagram
\begin{equation}
\label{E:ThreeCommutations}
\xymatrix{
\dtop 
  \ar[r] ^-{F}
  \ar[d] _-{(-)^\sharp}
& 
\opnatnu{\act}
  \ar[d] ^-{(-)^\flat} 
  \ar[r] ^-{\Er}
&
\cat
  \ar[d] ^-{(-)^o}
\\
\dtop 
  \ar[r] _-{F}
& 
\opnatnu{\act}
  \ar[r] _-{\Er}
&
\cat
}
\end{equation}
where the functor $\Er:\opnat{\act}\rightarrow \cat$ sends a pair $(\Cr,D,\nu)$ to the category in $\cat_{\Cr_0} / \Cr$ defined using the construction described in \ref{S:LaxSystemToInternalGroups} and \ref{S:SplitObject}.
The rightmost square commutes strictly. Indeed, denoting by $\Er_{(\Cr, D,\nu)}$ the category obtained from the natural system $(D,\nu)$ on the category $\Cr$, we have that $\Er_{(\Cr,D,\nu)^\flat}$ is the category with the same $0$-cells as $\Cr^o$ and in which $1$-cells are defined via the hom-sets
\[
\Er(y,x)= \coprod_{f^o \in \Cr^o(y,x)} D_f \, ,
\]
since by definition, $D^\flat_{f^o} = D_f$.
On the other hand, $\Er_{(\Cr,D,\nu)}$ has the same $0$-cells as $\Cr$ and $1$-cells are defined via the hom-sets
\[
\Er(x,y)= \coprod_{f\in \Cr(x,y)} D_f \, .
\]
Thus $\Er_{(\Cr,D,\nu)^\flat}$ coincides with $\Er_{(\Cr,D,\nu)}^o$. 
As consequence, if the leftmost square in diagram~\ref{E:ThreeCommutations} commutes up to isomorphism, then the outer square commutes up to isomorphism. We have thus proved the following result.
\begin{proposition}
\label{P:commutationThreeSquares}
Any functor $F:\dtop\fl\opnatnu{\act}$ which is time-reversal with respect to $\opnat{\act}$ can be extended into a time-reversal functor $\Er\circ F:\dtop\fl\cat$.
\end{proposition}

\subsection{Relative directed homotopy}
\label{subsec:relhom}

\subsubsection{Relative homotopy}
Let us recall the definition of the relative homotopy groups of a pair of topological spaces.
For $n\geq 1$, let $I^n$ denote the $n$-dimensional unit cube $[0,1]^n$. We single out the face $I^{n-1} := \set{(x_1, \dots , x_n) \hskip.2cm | \hskip.2cm x_n = 0}$, and define $J^{n}$ to be the closure of $\partial I^{n} \setminus I^{n-1}$ in $I^n$. Given a \emph{pointed pair} of topological spaces $(X,A)$, \ie a space $X$ and a subspace $A\subseteq X$ pointed at $x\in A$, we define, for $n\geq 1$, the $n^{th}$ \emph{relative homotopy} of $(X,A)$ by setting
$$\pi_{n}(X,A) := [\map{f}{(I^n, \partial I^n, J^n)}{(X,A,x)} ]$$
\ie the homotopy classes of maps $f:I^n \rightarrow X$ with $f(\partial I^n) \subseteq A$ and $f(J^n) = \set{x}$. The homotopies between such maps must satisfy the same conditions.

Note that for $n=1$, this is not a group for concatenation. Indeed, a map $\map{f}{(I, \set{0,1}, {1})}{(X,A,x)}$ is required to end at $x$, but can start anywhere in $A$, and therefore such maps can in general not be concatenated. We consider $\pi_{1}(X,A) $ as a pointed set, the pointed element being the class of paths $f$ such that $f$ is homotopic to a path $g$ with its image contained in $A$, \ie$g([0,1])\subseteq A$. For $n\geq 2$, $\pi_{n}(X,A)$ forms a group under concatenation, and is abelian for $n\geq 3$. For $f : I^n \fl X$, its class in $\pi_n(X,A)$ is the identity element if, and only if, it is homotopic to a map $g : I^n \fl X$ with its image contained in $A$.

 The assignment of relative homotopy groups to a pointed pair of spaces is functorial. Its domain is the category of pointed pairs of topological spaces, denoted by ${\top_*}_2$, in which a morphism $f:(X,A,x)\fl(Y,B,y)$ is a continuous map $f:X\fl Y$ such that $f(A)\subseteq f(B)$ and $f(x)=y$, and its codomain is $\psetcat$ for $n=1$, $\gp$ for $n=2$ and $\ab$ for $n\geq 3$. We can therefore consider these as functors with values in~$\act$ for all $n\geq 1$. 
 
\subsubsection{Relative directed homotopy}
 
Let us extend relative homotopy to dispaces. A \emph{directed subspace} of a dispace $\Xr=(X,dX)$ is a dispace $\Ar=(A,dA)$ such that $A\subseteq X$ is a subspace and $dA \subseteq dX$.
We define \emph{the category of pairs of dispaces}, denoted $\dtop_2$, as the category having objects $(\Xr , \Ar)$ with $\Ar$ a directed subspace of $\Xr$, and in which a morphism $\varphi : (\Xr, \Ar) \fl (\Yr, \Br)$ is a morphism $\varphi : \Xr \fl \Yr$ of dispaces such that $\varphi(A) \subseteq B$ and $\varphi_*(dA) \subseteq dB$.

For $n\geq 2$, the $n^{th}$ \emph{natural system of relative directed homotopy} associated to the pair $(\Xr,\Ar)$ is the natural system on $\diP{\Ar}$, denoted by $\sysp{n}{\Xr,\Ar}$, sending a dipath $\map{f}{x}{y}$ in $dA$ to the $(n-1)^{th}$ relative homotopy group of the pointed pair $(\trace{\Xr}{x}{y}, \trace{\Ar}{x}{y}, f)$, and whose group homomorphisms induced by extensions $(u,v)$ are defined by concatenation of paths as in \ref{SSS:TraceDiagrams}. 
Using a notion of relative trace diagrams $\dtop_2 \fl {\top_*}_2$ and similar arguments as those in Section~\ref{S:DirectedHomotopyInternalGroup}, it can be shown that, for each~$n\geq 2$, $\sysp{n}{\Xr,\Ar}$ extends to functors
\[
\overrightarrow{P}_{\!n} : \dtop_2 \fl \opnat{\act}.
\]

\subsubsection{Relative homotopy sequence in $\act$}
Grandis shows in {\cite[Theorem 6.4.9]{marco2013homological}} that given a pointed pair of topological spaces $(X,A)$, there is a long exact sequence in $\act$:
\begin{align*}
\cdots 
\: \rightarrow \:
\pi_n (A) \: \rightarrow \: \pi_n (X) \: \rightarrow \: & \pi_{n}(X,A) \: \overset{\partial_n}{\rightarrow} \: \pi_{n-1}(A) \: \rightarrow \: \cdots \\
\cdots \: \overset{v}{\rightarrow} \: \pi_1 (X) \: \overset{f}{\rightarrow} \: (\pi_1(X,A), \pi_1 (X)) & \: \overset{g}{\rightarrow} \: \pi_0(A) \: \overset{h}{\rightarrow} \: \pi_0(X) \: \rightarrow \: \pi_{0}(X,A) \: \rightarrow \: 0.
\end{align*}
Note that these assignment is functorial from ${\top_*}_2$ to the category of long exact sequences in $\act$.
All of the morphisms of this sequence are induced by inclusions, except the last non-trivial homomorphism and the homomorphisms $\partial_n$, which are given by restriction to the distinguished face: $\partial_{n} ([\sigma]) = [\sigma]_{| I^{n-1}}$. Also recall that we have not defined a relative homotopy group for $n=0$; the object $\pi_{0}(X,A)$ is defined to be $\Crok{h}$. It is thus the quotient of the set of path-connected components of $X$ obtained by identifying the components which intersect $A$.

All the terms above $\pi_1(X)$ are groups, and the existence of this long exact sequence in $\gp$ is well known. Since exactness is carried into $\act$, this induces that the sequence is exact in $\act$ up to this object.
As observed above, $\pi_1 (X,A)$ is not a group, but a pointed set. There is a right action of the group $\pi_1(X)$ on this pointed set given by concatenation; the elements of $\pi_1 (X,A)$ have ending point $0_X$, and so we can concatenate with elements of $\pi_1(X)$ on the right. The sequence is exact at $\pi_1 (X) = (|\pi_1(X)|, \pi_1(X))$ since the image of $v$ is precisely $\Kernel{f'}$; indeed, $Fix_{\pi_1 (X)}(0_{\pi_1(X,A)}) = \pi_1(A)$.
The map $g$ sends $(\tau, \sigma) \in (\pi_1(X,A), \pi_1 (X))$ to the path-connected component of $\tau(0)\in A$. We therefore view it is a pointed set map from $\pi_1(X,A)$ to $\pi_0(A)$. The sequence is exact at $(\pi_1(X,A), \pi_1 (X))$ because the antecedents under $g$ of the pointed element of $\pi_0(A)$, namely the component containing the base point $x$, are elements of the orbit of the pointed element $0$ of $\pi_1(X,A)$, \ie $0\cdot \pi_1(X)$. This coincides with the image of $f$ since it is defined by sending $\sigma\in \pi_1(X)$ to $(0\cdot\sigma, \sigma)$.
Lastly, exactness at $\pi_0(A)$ is a consequence of the inverse image under $h$ of the pointed element $[x]$ in $\pi_0(X)$ being exactly $\set{[x]}$, the pointed element in $\pi_0(A)$, since $h$ is induced by the inclusion $A \hookrightarrow X$. Furthermore, for $\tau\in\pi_1(X,A)$, $g(\tau)$ is necessarily in the same path connected component as $x$. Thus, the image of $g$ coincides with the kernel of $h$.

\subsubsection{Natural relative homotopy sequence}
We endow the category $\natsys{\diP{\Ar}}{\act}$ with the structure of a homological category by letting null morphisms be those natural transformations which are null component-wise in $\act$. A sequence of natural systems of actions is then \emph{exact} when it is point-wise exact in $\act$. As a consequence we obtain the following long exact sequence of natural homotopy systems:

\begin{theorem}
\label{T:MainTheoremB}
Let $\Xr$ be a dispace and $\Ar$ be a directed subspace of $\Xr$.
There is an exact sequence in $\natsys{\diP{\Ar}}{\act}$:
\begin{align*}
\cdots 
\: \rightarrow \: \sysp{n}{\Ar} \: \rightarrow \: \sysp{n}{\Xr}\: \rightarrow \: & \sysp{n}{\Xr,\Ar} \: \overset{\partial_n}{\rightarrow} \: \sysp{n-1}{\Ar} \: \rightarrow \: \cdots \\
\cdots \: \rightarrow \: \sysp{2}{\Ar} \: \overset{v}{\rightarrow} \sysp{2}{\Xr} \: \overset{f}{\rightarrow} \: (\sysp{2}{\Xr,\Ar}, \sysp{2}{\Xr}) & \: \overset{g}{\rightarrow} \: \sysp{1}{\Ar} \overset{h}{\rightarrow} \: \sysp{1}{\Xr} \: \rightarrow \: \sysp{1}{\Xr,\Ar} \: \rightarrow \: 0.
\end{align*}
\end{theorem}

\subsubsection{Dicontractible subspaces}
A dispace $\Xr$ is called \emph{dicontractible} if all its natural homotopy
functors $\sysp{n}{\Xr}$ are trivial, e.g. are constant functors into a singleton for $n=1$ or a trivial group for $n\geq 2$.
As a consequence, of Theorem~\ref{T:MainTheoremB}, if $\Ar$ is a dicontractible directed subspace of $\Xr$, then we have an isomorphism in $\natsys{\diP{\Ar}}{\gp}$
\[
\sysp{n}{\Xr} \simeq \sysp{n}{\Xr,\Ar},
\]
for all $n\geq 3$.
Note that when $(X,dX)$ is the geometric realization of a finite precubical set, the dicontractibility condition is equivalent to asking that all path spaces are contractible. 

\subsubsection{A long exact fibration sequence in directed topology}
Recall that a morphism $\varphi : \Xr \fl \Yr$ of dispaces induces a natural tranformation $\overrightarrow{\varphi}: \overrightarrow{T}_*(\Xr) \dfl \overrightarrow{T}_*(\Yr)$. We consider morphisms $p : \Er \fl \Br$ of dispaces such that each component $\overrightarrow{p}_e$ is a fibration, for every $e$ a dipath of $\Er$. 

Given such a morphism $p$, we define the associated \emph{natural system of fibres}, denoted $\overrightarrow{T}_*(\lF)$, as the natural system of pointed topological spaces on $\diP{\Er}$ which sends a dipath $e$ to
\[
\overrightarrow{T}(\lF)_e = \left(\overrightarrow{p}_e^{-1}(p(e)) \ , \ e \right).
\]
Now for each $1$-cell $e$ of $\diP{\Er}$, denote by $\sysp{n}{\lF}_{e}$ (resp.~$\sysp{n}{\Er,\lF}_{e}$) the homotopy group (resp.~relative homotopy group)
\[
\pi_{n-1}\left(\overrightarrow{T}(\lF)_{e} \right) \qquad(\text{resp.~}
\pi_{n-1}\left( \overrightarrow{T}(\Er)_{e},\overrightarrow{T}(\lF)_{e} \right)).
\]
These are natural systems on $\diP{\Er}$.
Furthermore, for each $e$ dipath of $\Er$, the sequence
\[
\overrightarrow{T}(\lF)_e \fl \overrightarrow{T}(\Er)_e \fl \overrightarrow{T}(\Br)_{p(e)}
\]
of topological spaces induces a long exact sequence of homotopy groups. Extending this to lower-dimensional homotopy groups via~{\cite[Theorem 6.4.9]{marco2013homological}} yields the following result.

\begin{theorem}
Let $p : \Er \fl \Br$ be a morphism of dispaces inducing (Serre) fibrations $\overrightarrow{p}_{\!e}$ for every $1$-cell $e$ of $\diP{\Er}$. Then we obtain a long exact sequence in $\natsys{\diP{\Er}}{\act}$:
\begin{align*}
\cdots 
\: \rightarrow \:
\sysp{n}{\lF} 
\: \rightarrow \:
\sysp{n}{\Er} 
\: \rightarrow \:
&
\: \sysp{n}{\Er,\lF}
\overset{}{\rightarrow} 
\sysp{n-1}{\lF}
\: \rightarrow \:
\cdots \\
\cdots 
\: \rightarrow \:
\sysp{2}{\lF}
\overset{}{\rightarrow} 
\sysp{2}{\Er}
\overset{}{\rightarrow} 
\left(\sysp{2}{\Er,\lF} \ , \ \sysp{2}{\Er}\right) 
& 
\: \overset{}{\rightarrow} \:
\sysp{1}{\lF} 
\overset{}{\rightarrow} 
\sysp{1}{\Er} 
\: \rightarrow \:
\sysp{1}{\Er,\lF} 
\: \rightarrow \: 0.
\end{align*}

Furthermore, $\sysp{n}{\Er,\lF} \cong p^* (\sysp{n}{\Br})$ for all $n\geq 2$. In particular, when $\overrightarrow{T}_*(\Br)_{p(e)}$ is path connected for all dipaths $e$ of $\Er$, the isomorphism holds for all $n\geq 1$.
\end{theorem}

\begin{small}
\renewcommand{\refname}{\Large\textsc{References}}
\bibliographystyle{abbrv}
\bibliography{LaxHomotopy}

\begin{thebibliography}{10}

\bibitem{BauesWirsching}
H.-J. Baues and G.~Wirsching.
\newblock Cohomology of small categories.
\newblock {\em Journal of Pure and Applied Algebra}, 38(2--3):187 -- 211, 1985.

\bibitem{BauesWirsching85}
H.-J. Baues and G.~Wirsching.
\newblock Cohomology of small categories.
\newblock {\em J. Pure Appl. Algebra}, 38(2-3):187--211, 1985.

\bibitem{dubutthesis}
J.~Dubut.
\newblock {\em Directed homotopy and homology theories for geometric models of
  true concurrency}.
\newblock PhD thesis, 09 2017.

\bibitem{naturalhomology}
J.~Dubut, E.~Goubault, and J.~Goubault{-}Larrecq.
\newblock Natural homology.
\newblock In {\em Automata, Languages, and Programming - 42nd International
  Colloquium, {ICALP} 2015, Kyoto, Japan, July 6-10, 2015, Proceedings, Part
  {II}}, pages 171--183, 2015.

\bibitem{Eilenberg}
J.~Dubut, {\'E}.~Goubault, and J.~Goubault{-}Larrecq.
\newblock Directed homology theories and {E}ilenberg-{S}teenrod axioms.
\newblock {\em Applied Categorical Structures}, pages 1--33, 2016.

\bibitem{fajstrup16}
L.~Fajstrup, E.~Goubault, E.~Haucourt, S.~Mimram, and M.~Raussen.
\newblock {\em Directed {A}lgebraic {T}opology and {C}oncurrency}.
\newblock Springer, 2016.

\bibitem{frgatc}
L.~Fajstrup, E.~Goubault, and M.~Raussen.
\newblock Algebraic topology and concurrency.
\newblock {\em Theoretical Computer Science}, 357(1-3):241--278, 2006.

\bibitem{gouhha}
E.~Goubault.
\newblock Some geometric perspectives in concurrency theory.
\newblock {\em Homology, Homotopy and Applications}, 5(2):95--136, 2003.

\bibitem{2017arXiv170905702G}
E.~{Goubault}.
\newblock {On directed homotopy equivalences and a notion of directed
  topological complexity}.
\newblock {\em ArXiv e-prints}, Sept. 2017.

\bibitem{CONCUR92}
E.~Goubault and T.~P. Jensen.
\newblock Homology of higher dimensional automata.
\newblock In {\em {CONCUR} '92, Third International Conference on Concurrency
  Theory, Stony Brook, NY, USA, August 24-27, 1992, Proceedings}, pages
  254--268, 1992.

\bibitem{CAT0}
E.~Goubault and S.~Mimram.
\newblock Directed homotopy in non-positively curved spaces.
\newblock {\em Submitted}, 2016.

\bibitem{grandisbook}
M.~Grandis.
\newblock {\em Directed Algebraic Topology, Models of non-reversible worlds}.
\newblock Cambridge University Press, 2009.

\bibitem{marco2013homological}
M.~Grandis.
\newblock {\em Homological Algebra: In Strongly Non-abelian Settings}.
\newblock World Scientific Publishing Company, 2013.

\bibitem{GuiraudMalbos12advances}
Y.~Guiraud and P.~Malbos.
\newblock Higher-dimensional normalisation strategies for acyclicity.
\newblock {\em Adv. Math.}, 231(3-4):2294--2351, 2012.

\bibitem{GuiraudMalbos10smf}
Y.~Guiraud and P.~Malbos.
\newblock Identities among relations for higher-dimensional rewriting systems.
\newblock In {\em O{PERADS} 2009}, volume~26 of {\em S\'emin. Congr.}, pages
  145--161. Soc. Math. France, Paris, 2013.

\bibitem{Leech85}
J.~Leech.
\newblock Cohomology theory for monoid congruences.
\newblock {\em Houston J. Math.}, 11(2):207--223, 1985.

\bibitem{MacLane98}
S.~Mac~Lane.
\newblock {\em Categories for the working mathematician}, volume~5 of {\em
  Graduate Texts in Mathematics}.
\newblock Springer-Verlag, New York, second edition, 1998.

\bibitem{Porter2}
T.~Porter.
\newblock Group objects in $cat_{\Sigma_0}/b$, 2012.
\newblock preprint.

\bibitem{Pratt}
V.~Pratt.
\newblock Modeling concurrency with geometry.
\newblock In {\em Proceedings of the 18th ACM SIGPLAN-SIGACT Symposium on
  Principles of Programming Languages}, POPL '91, pages 311--322. ACM, 1991.

\bibitem{quillencohomology}
D.~Quillen.
\newblock On the (co-) homology of commutative rings.
\newblock In {\em Applications of {C}ategorical {A}lgebra ({P}roc. {S}ympos.
  {P}ure {M}ath., {V}ol. {XVII}, {N}ew {Y}ork, 1968)}, pages 65--87. Amer.
  Math. Soc., Providence, R.I., 1970.

\bibitem{Quillen72}
D.~Quillen.
\newblock Higher algebraic {$K$}-theory. {I}.
\newblock {\em Lecture Notes in Mathematics}, 341:85--147, 1973.

\bibitem{Glabbeek}
R.~J. van Glabbeek.
\newblock On the expressiveness of higher dimensional automata.
\newblock {\em Theor. Comput. Sci.}, 356(3):265--290, 2006.

\bibitem{Bisimulation}
R.~J. van Glabbeek.
\newblock Bisimulation.
\newblock In {\em Encyclopedia of Parallel Computing}, pages 136--139. 2011.

\end{thebibliography}
\end{small}

\clearpage

\quad

\vfill

\begin{flushright}
\begin{small}
\noindent \textsc{Cameron Calk} \\
\url{cameron.calk@ens-lyon.fr} \\
{LIX, Ecole Polytechnique, CNRS, Universit\'e Paris-Saclay} \\
91128 Palaiseau, France \\
\end{small}
\end{flushright}

\bigskip

\begin{flushright}
\begin{small}
\noindent \textsc{Eric Goubault} \\
\url{goubault@lix.polytechnique.fr} \\
{LIX, Ecole Polytechnique, CNRS, Universit\'e Paris-Saclay} \\
91128 Palaiseau, France \\
\end{small}
\end{flushright}

\bigskip

\begin{flushright}
\begin{small}
\noindent \textsc{Philippe Malbos} \\
\url{malbos@math.univ-lyon1.fr} \\
Univ Lyon, Universit\'e Claude Bernard Lyon 1\\
CNRS UMR 5208, Institut Camille Jordan\\
43 blvd. du 11 novembre 1918\\
F-69622 Villeurbanne cedex, France
\end{small}
\end{flushright}

\vspace{0.25cm}

\begin{small}---\;\;\today\;\;-\;\;\hhmm\;\;---\end{small} \hfill
\end{document}